%% file: fine20.tex
\documentclass{siamltex}
\usepackage[latin1]{inputenc}
\usepackage{amsmath}
\usepackage{amsfonts}
\usepackage{amssymb}
\usepackage{hyperref}
\usepackage{graphicx}
\usepackage{wrapfig}
\usepackage{subfig}
\usepackage{float}

\newtheorem{remark}[theorem]{Remark}
\newtheorem{prop}[theorem]{Proposition}

\input{rateinit2.tex}

\begin{document}
\sloppy 

\pagestyle{myheadings}
\thispagestyle{plain}

\title{Simultaneous reconstruction of outer boundary shape and 
admittivity distribution in electrical impedance tomography}
\author{
J. Dard\'e\thanks{Department of Mathematics and Systems Analysis, 
Aalto University, P.O.~Box 11100, FI-00076 Aalto,
Finland ({\tt jeremi.darde@aalto.fi}). The work of J. Dard\'e  was 
supported by the Finnish Funding Agency for Technology and Innovation TEKES (contract 40370/08).} \and
N. Hyv\"onen\thanks{Department of Mathematics and Systems Analysis, 
Aalto University, P.O.~Box 11100, FI-00076 Aalto,
Finland ({\tt nuutti.hyvonen@aalto.fi}). The work of N. Hyv\"onen  was 
supported by the Finnish Funding Agency for Technology and Innovation TEKES (contract 40370/08) and the Academy of Finland (decision 135979).} \and 
A. Sepp\"anen\thanks{Department of Applied Physics, University of Eastern Finland, P.O.~Box 1627, FI-70211 Kuopio, Finland ({\tt aku.seppanen@uef.fi}). The work
of Aku Sepp\"anen was supported by the Academy of Finland (the Centre of Excellence in Inverse Problems Research and decisions 140280, 250215).} \and 
S. Staboulis\thanks{Department of Mathematics and Systems Analysis, 
Aalto University, P.O.~Box 11100, FI-00076 Aalto,
Finland ({\tt stratos.staboulis@aalto.fi}). The work of S. Staboulis  was 
supported by the Academy of Finland (decision 141044).}}

\maketitle

\begin{abstract}
The aim of electrical impedance tomography is to reconstruct the admittivity
distribution inside a physical body from boundary measurements of current
and voltage. Due to the severe ill-posedness of the underlying inverse problem,
the functionality of impedance tomography relies heavily on accurate 
modelling of the measurement geometry. In particular, almost all 
reconstruction algorithms require the precise shape of the imaged body 
as an input.
In this work, the need for prior geometric information
is relaxed by introducing a Newton-type output least squares
algorithm that reconstructs the admittivity distribution and the object
shape simultaneously. 
The method is built in the framework of the complete electrode model and it is
based on the Fr\'echet derivative of the corresponding current-to-voltage
map with respect to the object boundary shape. The functionality of the 
technique is demonstrated via numerical experiments with simulated measurement
data.
\end{abstract}


\begin{AMS}
{\sc 65N21, 35R30, 35J25}
\end{AMS}

\pagestyle{myheadings}

\markboth{J. DARD\'E, N. HYV\"ONEN, A. SEPP\"ANEN AND S. STABOULIS}
{RECONSTRUCTION OF OUTER BOUNDARY SHAPE IN EIT}

\section{Introduction}

{\em Electrical impedance tomography} (EIT) is a noninvasive imaging technique
which
has applications, e.g., in medical imaging, process tomography, and nondestructive
testing of materials \cite{Borcea02,Cheney99,Uhlmann09}. The objective of 
EIT is to reconstruct the admittivity distribution inside a
physical body $\Omega$ from boundary measurements of current and voltage. 
The most
accurate model for EIT is the {\em complete electrode
model} (CEM), which takes into account electrode shapes and 
contact impedances at electrode-object interfaces \cite{Cheng89}. 

A real-life measurement setting of EIT typically contains more unknowns than 
the mere admittivity distribution: The exact electrode locations, the 
contact impedances and the shape of the imaged object are not 
necessarily known accurately. (As an example, consider a medical 
application where the body shape and the contact impedances vary from 
patient to patient.) These kinds of inaccuracies comprise a considerable 
difficulty for establishing EIT as a practical imaging modality since it is
well known 
that even slight mismodelling can quite easily ruin the reconstruction of the 
admittivity \cite{Barber88, Breckon88, Kolehmainen97}. The problems resulting from the aforementioned model uncertainties have partly been resolved in earlier works: Two alternative ways to handle unknown contact impedances have been introduced in \cite{Nissinen09,Vilhunen02}, and fine-tuning the information on electrode positions has been considered in \cite{Darde12}. A brief review of the approaches to tackling the problem with an unknown object boundary shape is given in the following; for a more extensive discussion, see \cite{Nissinen11b}.

Undoubtedly the most common way to treat problems resulting from an inaccurately known boundary shape is the use of {\it difference imaging}, where the alteration in the admittivity distribution is reconstructed on the basis of the difference between EIT measurements corresponding to two time instants (or frequencies) \cite{Barber84}. The method is based on the idea that the modeling errors are partly removed when difference data are used --- given that the boundary shape remains unchanged between the two measurements. However, the difference imaging approach is highly approximative, because it relies on a linearization of the highly nonlinear forward model of EIT. Moreover, even if difference data are available, the boundary shape may also have changed between the measurements. This is the case, e.g., when imaging a human chest during a breathing cycle. A successful approach to coping with an unknown object boundary in {\it absolute} EIT imaging was suggested by Kolehmainen, Lassas and Ola \cite{Kolehmainen05,Kolehmainen07}. Their method is based on allowing slightly anisotropic conductivities and on the use of sophisticated mathematical instruments such as quasiconformal maps and Teichm\"uller spaces. In~\cite{Nissinen11} the so-called approximation error approach \cite{Kaipio05} was adapted to the compensation for errors resulting from an inaccurately known boundary shape in the framework of EIT. The approximation error method is based on the Bayesian inversion paradigm; the governing idea is to represent the error due to inaccurate modeling of the target as an auxiliary noise process. The (second order) statistics of the modeling error are approximated via simulations based on prior probability models for the admittivity and the boundary shape. The application of EIT to imaging of human thorax was considered in  \cite{Nissinen11}, where the approximated statistics of the modeling error were computed based on an atlas of anatomical CT chest images. In \cite{Nissinen11b}, the method was further developed to allow the reconstruction of the boundary shape. See also \cite{Tossavainen04,Tossavainen06}, where an optimization based technique was applied to the estimation of partially unknown boundary shape in  process tomography applications.

This work introduces an iterative Newton-type output least squares
algorithm that tolerates uncertainties in the geometry of 
the imaged object. To be more precise, our aim is to include
the estimation of the shape of the object boundary as a part of the
reconstruction method. The required Fr\'echet derivative of the
measurement map of the CEM with respect
to the exterior boundary shape is obtained with the help of domain derivative
techniques stemming from \cite{Kirsch93, Hettlich95, Hettlich98, Hyvonen10}; 
see also \cite{Delfour01} for a general theory of shape differentiation. 
However, unlike in \cite{Kirsch93, Hettlich95, Hettlich98, Hyvonen10}, 
the elliptic boundary value problem defining the derivative falls outside
the standard $H^1(\Omega)$-based variational theory due to Dirac delta type
boundary conditions on the edges of the electrodes. This difficulty is
tackled following the guidelines in \cite{Darde12}, where 
Fr\'echet derivatives with respect to electrode shapes were considered, 
resulting in
a well-posed `derivative problem' that is uniquely solvable 
in $H^{1-\epsilon}(\Omega)$, $\epsilon > 0$.

Our approach is made computationally more tractable by introducing 
a dual method for sampling the $H^{1-\epsilon}$-regular shape derivative; 
in particular, it turns out that the reconstruction algorithm can be 
implemented without having to solve any forward problems with distributional 
boundary conditions. This observation is concretized by the numerical
examples clearly demonstrating that the electrode measurements of 
EIT carry information on both the admittivity distribution 
and the object boundary shape.
The numerical studies are based on simulated measurement data and carried
out in three-dimensions, with the corresponding parameter choices founded on 
the Bayesian paradigm \cite{Kaipio05}.



This text is organized as follows. Section \ref{sec:CEM} recalls the CEM and 
its fundamental properties. The main 
Fr\'echet differentiability result is formulated in Section \ref{sec:Shape} and its proof
is given in Section \ref{sec:Proof}. Section \ref{sec:Algorithm} introduces the
reconstruction algorithm, which is then tested numerically in Section \ref{sec:Numerics}.
Finally, Section \ref{sec:Conclusion} lists the concluding remarks.

\section{Complete electrode model}
\label{sec:CEM}

Let $\Omega\subset{\R}^n$, $n=2$ or $3$, be a bounded domain and assume that its boundary ${\partial\Omega}$ is an orientable $C^\infty$-manifold. We denote by $\sigma \colon \Omega \to {\C}^{n\times n}$ the electrical admittivity distribution of $\Omega$ and assume that it satisfies the following, physically reasonable conditions \cite{Borcea02}: 
\begin{align}\label{sigma}
\sigma=\sigma^{\rm T},\qquad {\rm Re}\,(\sigma\xi\cdot\overline{\xi}) \geq C_1 |\xi|^2, \qquad |\sigma\xi\cdot\overline{\xi}| \leq C_2|\xi|^2 
\end{align} 
for some constants $C_1,C_2 > 0$ and for all $\xi\in{\C}^n$ almost everywhere in $\Omega$. 

Assume that the boundary ${\partial\Omega}$ is partially covered with 
$M \in {\N} \setminus \{1\}$ 
well-separated, open, bounded and connected electrodes $\{E_m\}_{m=1}^M$, i.e.,
\begin{equation}\begin{split}\label{E} 
E_m \subset {\partial\Omega},\;\; m= 1, \dots , M, \quad {\rm and} \quad \overline{E}_j\cap \overline{E}_k = \varnothing, \;\;\; j\neq k.
\end{split}\end{equation} 
The electrodes are modelled as ideal conductors. The union of the electrodes is denoted by $E = \cup_m E_m$, and the frequency domain representations of the time-harmonic electrode current and potential patterns by the vectors $I = [I_m]_{m=1}^M$ and $U = [U_m]_{m=1}^M$ of ${\C}^M$, 
respectively, where $I_m,U_m\in{\C}$ correspond to the measurements on the $m$th
electrode. 
Take note that the current vector $I$, actually, belongs to the subspace
\begin{equation}\begin{split} 
{\C}^M_\diamond := \bigg\{ [c_1,\ldots ,c_M] \in {\C}^M \ \Big| \ \sum_{m=1}^M c_m= 0 \bigg\} 
\end{split}\end{equation}
due to to the current conservation law.
The contact impedances (cf. \cite{Cheng89}) that characterize the thin and highly resistive layers at the electrode-object interfaces are modelled by $z\in {\C}^M$ that is assumed to satisfy
\begin{equation}\begin{split}
\label{z} {\rm Re}\, z_j > 0, \qquad j = 1, \dots , M.
\end{split}\end{equation}

According to the CEM \cite{Cheng89,Somersalo92}, the pair $(u,U) \in {\mathcal{H}^1}(\Omega) := (H^1(\Omega)\oplus {\C}^M)/{\C}$, composed of the electromagnetic potential within $\Omega$ and those on the electrodes, is the
unique solution of the elliptic boundary value problem
\begin{equation}
\label{cemeqs}
\begin{array}{ll}
\nabla \cdot \sigma\nabla u = 0 \qquad  &{\rm in}\;\;\Omega, \\[8pt] 
\nu\cdot\sigma\nabla u = 0 \qquad &{\rm on}\;\;{\partial\Omega}\setminus\overline{E},\\[8pt] 
u+z_m\nu\cdot\sigma\nabla u = U_m \qquad &{\rm on}\;\; E_m, \quad m=1, \dots, M, \\[4pt] 
{\displaystyle \int_{E_m}\nu\cdot\sigma\nabla u\,dS} = I_m, \qquad & m=1,\ldots,M, 
\end{array}
\end{equation}
for a given net electrode current pattern $I \in {\C}^M_\diamond$ and with 
$\nu = \nu(x)$ denoting the exterior unit normal of $\partial \Omega$.
The definition of ${\mathcal{H}^1}(\Omega)$ as a quotient space emphasizes the freedom in the
choice of the ground level of potential; in other words, one can never
measure absolute potentials, only potential differences.

The weak formulation of the CEM forward problem (\ref{cemeqs}) is to find $(u,U) \in {\mathcal{H}^1}(\Omega)$ that satisfies \cite{Somersalo92}
\begin{equation}\begin{split}\label{perus}
B \! \left\lbrace(u,U),(v,V)\right\rbrace = \sum_{m=1}^M I_m \overline{V}_m  \qquad \mbox{for all} \
(v,V) \in {\mathcal{H}^1}(\Omega),
\end{split}\end{equation}
where the sesquilinear form $B\colon {\mathcal{H}^1}(\Omega)\times{\mathcal{H}^1}(\Omega) \to {\C}$ 
is defined by 
\begin{equation}\begin{split}
\label{B} B \! \left\lbrace(u,U),(v,V)\right\rbrace=\int_{\Omega}\sigma\nabla u\cdot \nabla \overline{v} \,dx + 
\sum_{m=1}^M \frac{1}{z_m} \int_{E_m}( U_m -u ) ( \overline{V}_m -\overline{v} )\, dS. 
\end{split}\end{equation}
The form $B$ is concordant with the natural quotient topology of ${\mathcal{H}^1}(\Omega)$ 
(cf.~\cite[Corollary 2.6]{Hyvonen04}), i.e., for all $(u,U),(v,V)\in {\mathcal{H}^1}(\Omega)$
\begin{align*} 
\lve B\!\left\lbrace(u,U),(v,V)\right\rbrace\rve &\leq C_1\norm{(u,U)}_{{\mathcal{H}^1}(\Omega)}\norm{(v,V)}_{{\mathcal{H}^1}(\Omega)}, \\[2mm]
{\rm Re}\, B\!\left\lbrace(v,V),(v,V)\right\rbrace &\geq C_2\norm{(v,V)}^2_{{\mathcal{H}^1}(\Omega)},
\end{align*} 
where
$$ 
\norm{(v,V)}_{{\mathcal{H}^1}(\Omega)} := \inf_{c\in{\C}}\Big\{ \norm{v-c}^2_{H^1(\Omega)} + \sum_{m=1}^M \abs{V_m-c}^2 \Big\}^{1/2}.
$$ 
The unique solvability of (\ref{cemeqs}) follows by combining 
the above estimates and the obvious boundedness of the antilinear 
functional on the right-hand side
of (\ref{perus}) with the Lax--Milgram lemma \cite{Hyvonen04,Somersalo92}. 
This procedure also provides the estimate
\begin{equation}
\label{bbound}
\| (U,u) \|_{{\mathcal{H}^1}(\Omega)} \leq \ C |I|,
\end{equation}
where the constant of continuity $C = C(\Omega,\sigma,z)$ can be chosen 
independently of the electrodes if it is assumed that 
\begin{equation}
\label{e_ehto}
\min_{1\leq m \leq M} |E_m| \geq c
\end{equation}
for some constant $c>0$ (cf., e.g., \cite[(2.4)]{Hanke11}).
In the rest of this work, we make the assumption (\ref{e_ehto}) on the
considered electrode configurations implicitly. 

An ideal measurement corresponding to the CEM provides the electrode voltages 
$U \in {\C}^M / {\C}$ for some applied current pattern $I \in {\C}^M_\diamond$. For a given measurement setting $\bra{\Omega,E,\sigma,z}$, we thus define the measurement operator $R\colon {\C}^M_\diamond \to {\C}^M/{\C}$ by 
\begin{equation}\begin{split}\label{c-to-v}  
R: I \mapsto U. 
\end{split}\end{equation}
Obviously, $R$ is linear and bounded (cf.~(\ref{cemeqs}) and (\ref{bbound})), 
with a constant of continuity that can be chosen independently of the electrode 
configuration under the assumption~(\ref{e_ehto}).

To conclude this section, we note that for smooth $\sigma$ the interior 
potential has more regularity, namely
$$
\nu\cdot\sigma \nabla u|_E\in H^1(E), \quad \nder u|_{\partial\Omega}\in H_\diamond^{1/2-\epsilon}({\partial\Omega}), \quad u\in H^{2-\epsilon}(\Omega)/{\C}
$$ 
for all $\epsilon > 0$, as reasoned in \cite[Remark 1]{Darde12}. When
appropriate, we emphasize the last statement by writing $(u,U) \in 
\mathcal{H}^{2-\epsilon}(\Omega) := (H^{2-\epsilon}(\Omega)\oplus {\C}^M)/{\C}$.

\section{Shape derivative}
\label{sec:Shape}
In this section, we introduce the derivative of the CEM measurement map
with respect to perturbations of the object boundary $\partial \Omega$.
We begin by specifying how exactly the boundary is perturbed.

For $h\in C^1({\partial\Omega},{\R}^n)$ we define 
$$
F[h](x) = x + h(x), \qquad x \in {\partial\Omega},
$$
and use the abbreviation ${\partial\Omega}_h$ for the perturbed boundary, that is, 
$$
{\partial\Omega}_h = F[h]({\partial\Omega}) = \Set{y\in {\R}^n}{y=F[h](x)\;\;{\rm for\;some}\;\, x\in{\partial\Omega}}.
$$
The open, origin-centered ball of radius $d>0$ in the topology of $C^1({\partial\Omega},{\R}^n)$ is denoted by $\mathcal{B}_d$, i.e., 
$$
\mathcal{B}_d = \Set{h\in C^1({\partial\Omega},{\R}^n)}{\norm{h}_{C^1({\partial\Omega},{\R}^n)} < d}.
$$
Following \cite{Delfour01}, we introduce a special family of diffeomorphisms
of ${\R}^n$ to itself:
$$
\F_0^1 = \Set{F\colon {\R}^n \to {\R}^n}{F-{\rm id} \in C_0^1({\R}^n,{\R}^n)\;\,{\rm and}\;\,F^{-1}\in C^1({\R}^n,{\R}^n)},
$$
where $C_0^1({\R}^n,{\R}^n)$ denotes the space of continuously differentiable 
vector fields that together with their partial derivatives vanish at infinity.
In particular, when equipped with the natural norm, $C_0^1({\R}^n,{\R}^n)$ 
is a Banach space (cf.~\cite[p.~68]{Delfour01}). The following 
proposition lists some fundamental properties of $F[h] = {\rm id} + h$ for 
$h\in\mathcal{B}_d$ with small enough $d>0$. 

\begin{prop}\label{extension} 
There exists $d=d(\Omega) > 0$ such that the following hold:
\begin{itemize}\item[{\rm (a)}] For every $h\in\mathcal{B}_d$, $\partial \Omega_h$
is the boundary of a bounded $C^1$-domain $\Omega_h$, and
the mapping $F[h]$ is a $C^1$-diffeomorphism from ${\partial\Omega}$ onto ${\partial\Omega}_h${\rm ;}
\item[{\rm (b)}] There exists an extension operator $\mathcal{E}\colon \mathcal{B}_d \to C_0^1({\R}^n,{\R}^n)$ such that
$$
\mathcal{E}[h]|_{\partial\Omega} = h, \qquad \norm{\mathcal{E}[h]}_{C^1({\R}^n, {\R}^n)} 
\leq C(\Omega)\norm{h}_{C^1({\partial\Omega},{\R}^n)}
$$
and the extended mapping 
$$
F[\mathcal{E}[h]] = {\rm id} + \mathcal{E}[h]
$$
belongs to $\F_0^1$ for all $h \in\mathcal{B}_d$.
\end{itemize}
\end{prop}

\begin{proof}
The first part of the claim follows from an application of the implicit
function theorem in local coordinates on $\partial \Omega$. The second 
part can be deduced, e.g., 
by first forcing $h$ to zero in a tubular neighborhood of $\partial \Omega$
and then using similar arguments as on page 78 of \cite{Delfour01}.
\end{proof} 

If there is no danger of a confusion, we abuse the notation 
by denoting the extensions $\mathcal{E}[h]$ and $F[\mathcal{E}[h]]$ by the original symbols $h$ 
and $F[h]$, respectively. Moreover, we assume implicitly that $d > 0$ is as
introduced in Proposition~\ref{extension}.

Obviously, the measurement operator of the CEM may be considered as a map from 
$\mathcal{B}_d \times {\C}^M_\diamond$ to ${\C}^M/{\C}$, i.e., 
\begin{equation}
\label{memap}
R: (h, I) \mapsto U[h], 
\end{equation}
where $(u[h],U[h])\in {\mathcal{H}^1}(\Omega_h)$ is the unique solution of \eqref{cemeqs} 
when $\Omega$ is replaced by $\Omega_h$ and the
electrodes $E_m$ by $E_m^h := F[h](E_m) \subset \partial \Omega_h$, 
$m= 1, \dots, M$.
To make this definition unambiguous and to simplify the analysis that follows,
we assume that 
$\sigma \in C^\infty({\R}^n, {\C}^{n\times n})$ with the bounds \eqref{sigma}
satisfied everywhere in $\R^n$, i.e., that the admittivity 
distribution is defined in everywhere in ${\R}^n$ --- or at least in some proper
neighborhood of $\Omega$. As a further simplification, we also assume that 
(in the three-dimensional case) the 
electrode boundaries $\partial E_m$, $m=1, \dots, M$, are smooth curves.


We denote by $h_\tau$ and $h_\nu$ the tangential 
(vector) and normal (scalar) components of $h\in\mathcal{B}_d$, respectively, 
that is, we have the (unique) decomposition $h = h_\tau + h_\nu\nu$. One
might expect that it is enough to consider 
perturbations that belong to the normal bundle of the boundary, i.e.,
ones that have vanishing tangential components. However, this turns out to
be a false intuition, because tangential vector fields typically affect
the measurement map in the `first order' by moving the
electrodes (cf.~\cite{Darde12}) --- even though they only define `second order' 
perturbations of the object boundary $\partial \Omega$ itself.

\begin{theorem}
\label{diffbound}
Under the above assumptions, the operator 
$R\colon  \mathcal{B}_d \times {\C}^M_\diamond \to {\C}^M/{\C}$ is Fr\'echet differentiable at 
the origin with respect to the first variable, i.e., there exists a
bounded bilinear operator $R': C^1({\partial\Omega},{\R}^n) \times {\C}^M_\diamond \to {\C}^M/{\C}$ such that 
$$
\lim_{h\to 0}\frac{1}{\| h \|_{C^1(\partial \Omega,{\R}^n)}} 
\|R[h] - R[0] - R'h \|_{\mathcal{L}({\C}^M_\diamond, {\C}^M/{\C})} = 0
$$
in $C^1(\partial \Omega, {\R}^n)$.
\end{theorem}

In the following, we will prove Theorem \ref{diffbound} in three dimensions, 
i.e.~for $n=3$,
which is the more challenging case. The two-dimensional counterpart can be 
obtained by following a similar line of reasoning.

The derivative $R'$ of Theorem \ref{diffbound} can, in fact, be given 
explicitly.
To this end, let $H \in C^\infty({\partial\Omega})$ be the mean curvature function defined
so that it is positive if the surface turns away from the exterior unit normal,
and consider the bounded surface divergence operator (cf., e.g., 
\cite{Colton92})
\begin{equation}
\label{Div}
{\rm Div}\colon [H^s({\partial\Omega})]^n_\tau \to H^{s-1}({\partial\Omega}), \qquad s \in {\R},
\end{equation}
with the weak definition
$$
\langle {\rm Div} \, v, \varphi \rangle_{\partial\Omega} = 
- \langle v,  {\rm Grad} \,\varphi  \rangle_{\partial\Omega},  \qquad \varphi \in 
C^{\infty}(\partial \Omega), 
$$ 
where $\text{Grad}$ denotes the surface gradient (cf., e.g., \cite{Delfour01}).
We also introduce a family of distributions $\{ \delta_m \}_{m=1}^M \subset
H^{-1/2 - \epsilon}({\partial\Omega})$, $\epsilon > 0$, defined through
$$
\langle \delta_{m}, v \rangle_{\partial \Omega} = \int_{{\partial E}_m} v\,ds,
\qquad v \in H^{1/2 + \epsilon}({\partial\Omega}),
$$
for $m=1, \dots , M$. Notice that any $v \in  H^{1/2 + \epsilon}({\partial\Omega})$, $\epsilon > 0$, 
has 
a well defined restriction $v|_{{\partial E}} \in  H^{\epsilon}({\partial E})$ due to 
the trace theorem, and thus the definition of 
the family $\{ \delta_m \}_{m=1}^M$ is unambiguous.
Moreover, we denote the characteristic function of $E_m \subset \partial \Omega$
by $\chi_m$, $m=1, \dots, M$, and the unit exterior normal of ${\partial E}$ in the 
tangent bundle of ${\partial\Omega}$ by $\nu_{{\partial E}}$.

With these tools in hand, let us consider the boundary value problem 
\begin{equation}
\label{bndeqs} 
\begin{array}{ll} 
\nabla\cdot\sigma\nabla u' = 0 & \quad  {\rm in}\ \Omega, \\[1mm] 
{\displaystyle \nu\cdot \sigma \nabla u' - \sum_{m=1}^M \frac{1}{z_m} (U' - u') \chi_m 
= f_1 + \sum_{m=1}^m \frac{1}{z_m}( f_2 \chi_m + f_3 \delta_m)} & \quad {\rm on} \ {\partial\Omega}, 
\\[1mm] \displaystyle{\int_{E_m} (U'_m - u') \, dS = - \int_{E_m} f_2 \,dS -
\int_{\partial E_m} f_3 \,ds}, & 
\quad  m=1,\ldots, M. 
\end{array}
\end{equation} 
Here, the inputs $f_1 \in  H^{-1/2-\epsilon}(\partial \Omega)$, $f_2
\in H^{1/2 -\epsilon}(E)$ and $f_3 \in H^{1-\epsilon}(\partial E)$ 
are defined with the help of $(u,U)\in \mathcal{H}^{2-\epsilon}(\Omega)$, i.e., 
the unperturbed solution of (\ref{cemeqs}):
\begin{align*}
f_1  &=  \, {\rm Div}(h_\nu(\sigma\nabla u|_{\partial \Omega})_\tau), \\[1mm]
f_2|_{E_m} &= \, h_\nu \Big( (n-1) H (U_m-u) - \frac{\partial u}{\partial\nu} \Big)\Big|_{E_m}, \\[1mm]
f_3|_{\partial E_m} &= \, (h \cdot \nu_{\partial E}) (U_m - u)|_{\partial E_m}.
\end{align*}
Notice that the claimed regularity of $f_1$, $f_2$ and $f_3$ follows 
from (\ref{Div}) and (consecutive) applications of the trace theorem.
It turns out that problem (\ref{bndeqs}) is uniquely solvable in 
$\mathcal{H}^{1-\epsilon}(\Omega)$, and that the corresponding solution defines the 
Fr\'echet derivative of Theorem \ref{diffbound}.

\begin{theorem}
\label{main}
Under the assumptions of Theorem \ref{diffbound}, the boundary value
problem \eqref{bndeqs} has a unique solution $(u'[h], U'[h])\in 
\mathcal{H}^{1-\epsilon}(\Omega)$, $\epsilon > 0$, for any
$h \in C^1({\partial\Omega},{\R}^n)$. Moreover, the Fr\'echet derivative
of Theorem \ref{diffbound}, i.e., $R':  C^1({\partial\Omega},{\R}^n) \times 
{\C}^M_\diamond \to {\C}^M/{\C}$, is given by
$$
R': (h, I) \mapsto U'[h].
$$
\end{theorem}
At first sight it may seem that Theorem \ref{main} is not very practical 
as it defines the Fr\'echet derivative of $R$ with the help of a
 boundary value problem that falls outside the $H^1$-based variational theory. 
Fortunately, there also exists a dual approach for sampling the shape
derivative.
\begin{corollary}\label{sample} Let $(\tilde{u} ,\tilde{U}) \in 
\mathcal{H}^{2-\epsilon}(\Omega)$, $\epsilon > 0$, be 
the solution of \eqref{cemeqs} for some electrode current pattern 
$\tilde{I}\in{\C}_\diamond^M$. Then, for any 
$(h,I)\in C^1({\partial\Omega},{\R}^n) \times {\C}_\diamond^M$ it holds that 
\begin{align}
\label{sampling}
\sum_{m=1}^M(R'(h,I))_m\tilde{I}_m = &-\int_{{\partial\Omega}}h_\nu (\sigma\nabla u)_\tau \cdot
(\nabla\tilde{u})_\tau \, dS \\
& - \sum_{m=1}^M\frac{1}{z_m}\int_{E_m}h_\nu\Big( (n-1)(U_m-u)H - \frac{\partial u}{\partial\nu}\Big)(\tilde{U}_m-\tilde{u}) \, dS \nonumber \\
& - \sum_{m=1}^M \frac{1}{z_m}\int_{\partial E_m}(h \cdot \nu_{\partial E})(U_m-u)
(\tilde{U}_m-\tilde{u}) \,ds, \nonumber
\end{align}
where $(u,U) \in \mathcal{H}^{2-\epsilon}(\Omega)$ is the solution of 
\eqref{cemeqs}.
\end{corollary}

\section{Proof of the main result}
\label{sec:Proof}
Before moving on to prove Theorems \ref{diffbound} and \ref{main} and Corollary \ref{sample}, we give a brief summary of the variational technique on which the proof is based. A more complete reasoning in a slightly different framework can be found in \cite[Section 5.1]{Darde12}. 

Let $F = F[h] = {\rm id} + h \in \F_0^1$, with $h \in \B_d$, be as in the previous section.
We introduce a pullback operator $F^\ast \colon H^1(\Omega_h) \to H^1(\Omega)$ defined by $F^\ast v = v\0 F|_{\Omega}$; it is easy to see that $F^\ast$ is a linear 
isomorphism. A simple change of variables
applied to the variational equation defining $(u[h], U[h]) \in \mathcal{H}^1(\Omega_h)$, cf.~\eqref{memap}, shows that
the difference of the pullback pair $(F^*u[h], U[h]) \in {\mathcal{H}^1}(\Omega)$ 
and the unperturbed solution $(u,U) = (u[0], U[0]) \in {\mathcal{H}^1}(\Omega)$ satisfies (cf.,~e.g.,~\cite{Hyvonen10})
\begin{align} 
B\{ (F^\ast u[h] - u, & \, U[h] - U),(v,V) \} \nonumber \\[2pt] 
= & \, \int_\Omega \pa{ \sigma - \sigma^\ast[h]}\nabla (F^\ast u[h])\cdot \nabla\overline{v} \,dx \nonumber \\[2pt] 
& + \sum_{m=1}^M\frac{1}{z_m}\int_{E_m} (U_m[h] - F^\ast u[h])(\overline{V}_m - \overline{v})(1 - |{\rm Jac}\,F|) \, dS \label{diffvar} 
\end{align}
for all $(v,V)\in \mathcal{H}^1(\Omega)$. Here, the pullback admittivity 
$\sigma^\ast[h]$ is defined as
$$
\sigma^\ast[h] = \abs{J_F}J_F^{-1} (F^\ast \sigma)(J_F^{-1})^{\rm T},
$$
$J_F$ is the Jacobian matrix of $F$, $|J_F|$ is the absolute value of its determinant, and ${\rm Jac}\,F$ is the surface Jacobian determinant of the restriction $F|_{\partial \Omega} : \partial \Omega \to
\partial \Omega_h$.
Moreover, it follows from the perturbation analysis in \cite{Hettlich95,Hettlich98} that
modulo $O (\norm{h}_{C^1({\R}^n, {\R}^n)}^2)$ it holds that
\begin{align}
\sigma - \sigma^\ast[h] &= \, \sigma J_h^{\rm T} + J_h \sigma - (h\cdot \nabla + \nabla \cdot h)\sigma,  \\[4pt]
1 - |{\rm Jac}\,F| &= \, - (n-1)Hh_\nu - {\rm Div} \, h_\tau,
\end{align}
where $h\cdot \nabla \sigma$ is defined as the matrix $(h\cdot \nabla\sigma_{ij})_{i,j = 1}^n$. In consequence, in order to estimate the difference 
$(F^\ast u[h] - u,  U[h] - U)$, it seems reasonable to consider the bounded 
sesquilinear functional $\Lambda: C^1({\R}^n, {\R}^n) \times {\mathcal{H}^1}(\Omega) \to {\C}$ defined 
by (cf.~\cite[(19)]{Darde12})
\begin{equation*}\begin{split}
\Lambda[h](v,V) & = \int_\Omega \pa{\sigma J_h^{\rm T} + J_h \sigma - (h\cdot\nabla + \nabla\cdot h)\sigma}\nabla u\cdot \nabla\overline{v}\, dx \\
&\qquad - \sum_{m=1}^M\frac{1}{z_m}\int_{E_m} (U_m-u)(\overline{V}_m-\overline{v})((n-1)Hh_\nu + {\rm Div} \, h_\tau) \, dS, 
\end{split}\end{equation*}
and the corresponding $h$-parametrized variational problem
\begin{equation}\label{varprob}
B\!\left\lbrace(w[h],W[h]),(v,V)\right\rbrace = \Lambda[h](v,V) \qquad {\rm for} \ {\rm all} \  (v,V)\in\mathcal{H}^1(\Omega),
\end{equation}
which has a unique solution in $\mathcal{H}^1(\Omega)$ due to the Lax--Milgram lemma.
The following proposition is a straightforward variation of 
\cite[Proposition 5.4]{Darde12}.
\begin{prop}
\label{hderivative}
Let $(u,U) \in {\mathcal{H}^1}(\Omega)$ and $(u[h], U[h]) \in\mathcal{H}^1(\Omega_h)$ be as defined in
Section \ref{sec:Shape} and $(w[h],W[h])\in \mathcal{H}^1(\Omega)$ the unique solution of \eqref{varprob}. Then, the estimate
$$
\norm{(F^\ast u[h] - u, U[h] - U) - (w[h],W[h])}_{\mathcal{H}^1(\Omega)} \leq C|I|\norm{h}_{C^1(\partial \Omega, {\R}^n)}^2
$$
holds with a constant $C>0$ that can be chosen independently of $I\in {\C}_\diamond^M$ and $h\in \mathcal{B}_d$.
\end{prop} 

Although the map $h \mapsto W[h]$ is a first order 
approximation of $h \mapsto (R[h] - R[0]) I = U[h] - U[0]$ around the origin, it does 
not provide a satisfactory definition for the Fr\'echet derivative $h \mapsto 
R'[h]$. Indeed, although $h \mapsto W[h]$ is clearly linear with
respect to the extension $h = \mathcal{E}[h] \in C^1_0({\R}^n, {\R}^n)$, it is not 
self-evident that the same also holds for the original perturbation
$h \in  C^1(\partial \Omega, {\R}^n)$ as required by Theorem~\ref{diffbound}. 
Moreover, from the 
computational view point, the extension of $h$ to the whole of ${\R}^n$ is
a nuisance that one wants to avoid. 

To get rid of this problem, 
we proceed as in \cite{Darde12} and modify the first component of $(w[h],W[h])$ in an appropriate way. This procedure involves including a directional derivative of the interior potential component of the unperturbed solution $(u,U) \in \mathcal{H}^{2-\epsilon}(\Omega)$ as an argument of the sesquilinear form $B\colon \mathcal{H}^1(\Omega)\times \mathcal{H}^1(\Omega) \to {\C}$. Since the derivatives of $u$ are merely in $H^{1-\epsilon}(\Omega)$ such analysis cannot be carried out without any modifications. For this reason, we introduce a sequence of smooth approximations for $u\in H^{2-\epsilon}(\Omega)/{\C}$, and subsequently also for $(w[h],W[h])\in \mathcal{H}^1(\Omega)$. As in \cite[Subsection 5.4]{Darde12}, we may pick a sequence $(u^{(j)},U^{(j)}) \in (C^\infty(\overline\Omega)\oplus \C^M)/\C$, $j=1,2, \dots$, 
such that 
\begin{equation}\label{jlimit}
\nabla\cdot\sigma\nabla u^{(j)} = 0 \quad {\rm in} \ \Omega \qquad {\rm and} \qquad
\lim_{j\to\infty} (u^{(j)}, U^{(j)}) = (u,U) \quad {\rm in} \ \mathcal{H}^{2-\epsilon}(\Omega).
\end{equation}
Moreover, we define $(w^{(j)}[h],W^{(j)}[h])\in \mathcal{H}^1(\Omega)$ to be the unique element of $\mathcal{H}^1(\Omega)$ that solves the variational problem
\begin{equation}\label{Lj}
B\{(w^{(j)}[h],W^{(j)}[h]),(v,V)\} = \Lambda_j[h](v,V) \qquad {\rm for \ all} \ (v,V)\in\mathcal{H}^1(\Omega),
\end{equation}
where the antilinear functional $\Lambda_j[h]\colon \mathcal{H}^1(\Omega) \to {\C}$ is defined 
via replacing $(u,U)$ by $(u^{(j)},U^{(j)})$ in the definition of $\Lambda[h]$. Through a slight 
variation of the argument in the proof of \cite[Lemma 5.7]{Darde12}, one 
easily obtains that
\begin{equation}\label{wjtow}
\lim_{j\to\infty} (w^{(j)}[h],W^{(j)}[h]) = (w[h],W[h]) \qquad {\rm in} \ \mathcal{H}^1(\Omega).
\end{equation} 

We proceed by defining the `augmented interior derivatives' by 
\begin{equation}\label{modified}
\tilde{w}[h] = w[h] - h\cdot \nabla u, \qquad \tilde{w}^{(j)}[h] = w^{(j)}[h] - h\cdot\nabla u^{(j)}, \quad j=1,2, \dots \ .
\end{equation}
Due to \eqref{jlimit} and \eqref{wjtow}, it follows that 
\begin{equation}
\label{wtiljtowtil}
\lim_{j\to\infty} (\tilde{w}^{(j)}[h],W^{(j)}[h]) = (\tilde{w}[h],W[h]) \qquad {\rm in} \ \mathcal{H}^{1-\epsilon}(\Omega)
\end{equation}
for any $\epsilon > 0$ (cf.~\cite{Lions72}). In the following, we will show that $(\tilde{w}[h],W[h])\in \mathcal{H}^{1-\epsilon}(\Omega)$ is the unique solution of \eqref{bndeqs} for $h\in\mathcal{B}_d$. In particular, the pair $(\tilde{w}[h],W[h])$ turns out to be 
independent of the extension of $h \in C^1(\partial \Omega, {\R}^n)$ to the
whole of ${\R}^n$.

{\em Proof of Theorems \ref{diffbound} and \ref{main}}.
In the first part of the proof, we show that the derivative problem 
(\ref{bndeqs}) is uniquely solvable, with the corresponding solution
being the above constructed pair $(\tilde{w}[h],W[h]) \in \mathcal{H}^{1-\epsilon}(\Omega)$ 
if $h \in \mathcal{B}_d$. First of all, we note that the uniqueness of the solution can be 
proved in exactly the same way as in the first part of the proof of 
\cite[Theorem 6.1]{Darde12}, which means that we can focus 
solely on the existence. In case that $\nu \cdot h \equiv 0$ and 
$h \in \mathcal{B}_d$, the fact that $(\tilde{w}[h],W[h])$ is a solution to 
(\ref{bndeqs}) follows from the second and third parts of the proof of 
\cite[Theorem 6.1]{Darde12}. Due to the linearity of the right-hand side 
of (\ref{bndeqs}) with respect to $h$, the unique solvability of (\ref{bndeqs}) thus follows if
we are able to show that $(\tilde{w}[h],W[h])$ is a solution also if 
$h = h_\nu \nu \in \mathcal{B}_d$. (For a general $h \notin \mathcal{B}_d$ the unique 
solution of (\ref{bndeqs}) can then be obtained by rescaling the 
corresponding solution for $h/c \in \mathcal{B}_d$ with large enough $c > 0$.)

(1) Assume that $h = h_\nu \nu \in \mathcal{B}_d$ and let us consider what kinds of variational equations $(\tilde{w}^{(j)}[h],W^{(j)}[h])\in \mathcal{H}^{1}(\Omega)$ satisfies. For now, let $\varphi\in C^\infty(\overline\Omega)$ and $V\in{\C}^M$ be arbitrary. Recalling first \eqref{modified} and the definition of $(w^{(j)}[h], W^{(j)}[h])\in\mathcal{H}^1(\Omega)$, and then using standard vector calculus, the first part of (\ref{jlimit}) 
and the divergence theorem, we obtain (cf., e.g., \cite{Hyvonen10})
\begin{align} 
\label{vareq}
B\big\{(\tilde{w}^{(j)}[h],W^{(j)}[h]),(\varphi, V)\big\} &= \int_{\partial\Omega} h_\nu\nu\cdot
\Big(\sigma\nabla u^{(j)}\frac{\partial \overline{\varphi}}{\partial\nu} - (\sigma\nabla u^{(j)}\cdot \nabla\overline{\varphi})\nu\Big) dS \nn \\ 
&\qquad + \sum_{m=1}^M\frac{1}{z_m}\int_{E_m}h_\nu\frac{\partial u^{(j)}}{\partial\nu}(\overline{V}_m -\overline{\varphi})\, dS  \\ 
&\qquad - \sum_{m=1}^M \frac{n -1}{z_m}\int_{E_m} h_\nu H (U^{(j)}_m -u^{(j)})(\overline{V}_m -\overline{\varphi})\,dS. \nn 
\end{align}
By dividing $\nabla u^{(j)}$ and $\nabla \varphi$ into tangential and normal
components on $\partial \Omega$, the first integrand further simplifies as
$$
 h_\nu\nu\cdot
\Big(\sigma\nabla u^{(j)}\frac{\partial \overline{\varphi}}{\partial\nu} - (\sigma\nabla u^{(j)}\cdot \nabla\overline{\varphi})\nu\Big) = -  h_\nu (\sigma\nabla u^{(j)})_\tau \cdot (\nabla\overline{\varphi})_\tau .
$$
Due to \eqref{jlimit} and the regularity of the unperturbed solution 
$(u,U) \in \mathcal{H}^{2-\epsilon}(\Omega)$, we may take the limit $j\to\infty$ in \eqref{vareq}, yielding  
\begin{align}
\label{varformula}
\lim_{j\to\infty}B\big\{ (\tilde{w}^{(j)}[h],W^{(j)}[h]),(\varphi,V) \big\} = &- \int_{\partial\Omega} h_\nu (\sigma\nabla u)_\tau \cdot (\nabla\overline{\varphi})_\tau \, dS \nonumber \\
&+ \sum_{m=1}^M\frac{1}{z_m}\int_{E_m}h_\nu\frac{\partial u}{\partial\nu}(\overline{V}_m-\overline{\varphi})\,dS \\
& -\sum_{m=1}^M\frac{n-1}{z_m} \int_{E_m} h_\nu H (U_m-u)(\overline{V}_m-\overline{\varphi}) \,dS.
\nonumber
\end{align}

To prove the first equality of \eqref{bndeqs}, let $V=0$ and $\varphi\in C^\infty_0(\Omega)$ be arbitrary. According to the definition of the sesquilinear form $B$, the identity \eqref{vareq} and the definition of distributional differentiation 
(cf., e.g., \cite{Dautray88}), it holds that 
\begin{equation}\label{disteq}
\dual{\nabla\cdot\sigma\nabla \tilde{w}^{(j)}[h],\varphi}_\Omega = 0.
\end{equation}
As the elliptic differential operator $\nabla\cdot\sigma\nabla \colon H^{1-\epsilon}(\Omega)/{\C} \to H^{-1-\epsilon}(\Omega)$ is continuous for any $\epsilon\in{\R}$ such that $\epsilon-1/2\notin \Z$ \cite[Chapter 1, Proposition 12.1]{Lions72}, we may pass the limit inside the brackets of \eqref{disteq}. Consequently, $\nabla\cdot\sigma\nabla \tilde{w}[h] = 0$ is satisfied in the sense of distributions in $\Omega$.

Assume next that $\varphi\in C^\infty(\overline\Omega)$ and let still $V=0$. The (generalized) Green's formula (cf. \cite[p.~382, Corollary 1]{Dautray88}) and 
\eqref{disteq} indicate 
\begin{equation*}
\begin{split}
B\big\{ (\tilde{w}^{(j)}[h],W^{(j)}[h]), (\varphi,0) \big\} &= \dual{\nu\cdot\sigma\nabla\tilde{w}^{(j)}[h],\overline{\varphi}}_{\partial\Omega}\\
&\qquad +\sum_{m=1}^M \frac{1}{z_m}\int_{E_m} (\tilde{w}^{(j)}[h] - W_m^{(j)}[h])\overline{\varphi}\,dS.
\end{split}
\end{equation*}
Moreover, according to \cite[Chapter 2, Theorem 7.3]{Lions72}, the Neumann trace map $v\mapsto \nu\cdot\sigma\nabla v|_{\partial\Omega}$ is well-defined and bounded from the closed subspace 
$$
\Set{v\in H^{1-\epsilon}(\Omega)/{\C}}{\nabla\cdot\sigma\nabla v = 0} \subset H^{1-\epsilon}(\Omega)/{\C}
$$
to $H^{-1/2 - \epsilon}({\partial\Omega})$. Thus, \eqref{wtiljtowtil} and the trace theorem give 
\begin{equation*}
\begin{split}
\lim_{j\to\infty}B\big\{ (\tilde{w}^{(j)}[h],W^{(j)}[h]), (\varphi,0) \big\} &= \dual{\nu\cdot\sigma\nabla\tilde{w}[h],\overline{\varphi}}_{\partial\Omega}\\
&\qquad +\sum_{m=1}^M \frac{1}{z_m}\int_{E_m} (\tilde{w}[h] - W_m[h])\overline{\varphi}\,dS.
\end{split}
\end{equation*}
On the other hand, by \eqref{varformula} it also holds that 
\begin{equation*}
\begin{split}
\lim_{j\to\infty}B\big\{ (\tilde{w}^{(j)}[h],W^{(j)}[h]), (\varphi,0) \big\} = &- \int_{\partial\Omega} h_\nu (\sigma\nabla u)_\tau \cdot (\nabla\overline{\varphi})_\tau \, dS \nonumber \\
&- \sum_{m=1}^M \frac{1}{z_m}\int_{E_m} h_\nu\frac{\partial u}{\partial\nu} \overline{\varphi} \, dS\\ 
&+\sum_{m=1}^M\frac{n-1}{z_m}\int_{E_m}h_\nu H (U_m-u) \overline{\varphi} \,dS. \nonumber
\end{split}
\end{equation*}
As $(\nabla\overline{\varphi})_\tau = {\rm Grad}\, \overline{\varphi}$ on $\partial \Omega$ (cf., e.g., \cite{Colton92}) and $C^\infty(\Omega)|_{\partial\Omega}$ is dense in $H^s({\partial\Omega})$ for any $s\in{\R}$, this proves that $(\tilde{w}[h],W[h])$ satisfies the second equation of \eqref{bndeqs}, since 
${\displaystyle h \cdot \nu_{\partial E} = 0}$ by assumption.

To prove the remaining, third condition of (\ref{bndeqs}), let $V$ be the $m$th 
coordinate vector and choose $\varphi \equiv 0$. Using the definition of 
$B$, \eqref{wtiljtowtil} and \eqref{vareq}, we conclude that 
\begin{equation}
\int_{E_m}(W_m[h] - \tilde{w}[h])dS = -\int_{E_m} h_\nu\Big((n-1)H(U_m-u) - \frac{\partial u}{\partial\nu} \Big)dS.
\end{equation}
Since $m$ was chosen arbitrarily, it follows that $(\tilde{w}[h],W[h])\in \mathcal{H}^{1-\epsilon}(\Omega)$ is a solution to \eqref{bndeqs} for $h \in \mathcal{B}_d$.	 

(2) Let us then prove that the mapping 
$$
R': (h, I) \mapsto U'[h], \quad C^1(\partial \Omega, {\R}^n) \times {\C}^M_\diamond 
\to {\C}^M / {\C}, 
$$
really defines the Fr\'echet derivative of Theorem \ref{diffbound} as claimed
in Theorem \ref{main}. First of all, it is easy to see that $R'$ is bilinear
since the right-hand side of (\ref{bndeqs}) depends bilinearly on 
$h \in C^1(\partial D, {\R}^n)$ and 
$(u,U)$, and the unperturbed solution $(u,U)$ itself depends linearly on the 
applied current pattern. Moreover, due to Proposition \ref{hderivative}
and since $U'[h] = W[h]$ for $h \in \mathcal{B}_d$ by the first part of the proof, we may estimate as follows:
$$
\| U[h] - U - U'[h] \|_{\C^M / \C} \leq C |I| \|h \|_{C^1(\partial \Omega, {\R}^n)}^2, \qquad
h \in \mathcal{B}_d,
$$
which completes the proof as $C>0$ can be chosen independently of 
$I \in {\C}^M_\diamond$ and $h \in \mathcal{B}_d$. \hfill $\Box$

We complete this section by providing a proof for Corollary \ref{sample}.

{\em Proof of Corollary \ref{sample}}. 
As in the previous proof, it is enough to consider 
small $h$ in the normal bundle of $\partial \Omega$, i.e., 
$h = h_\nu \nu \in \mathcal{B}_d$,
by the virtue of the linearity of the claimed sampling formula with respect to $h$ and the fact 
that for tangential perturbations the assertion follows through the same 
line of reasoning as \cite[Corollary 4.2]{Darde12}.

Let $\tilde{I}\in{\C}_\diamond^M$ and $(\tilde{u},\tilde{U})\in\mathcal{H}^1(\Omega)$ be as in Corollary \ref{sample} and consider $h = h_\nu \nu \in \mathcal{B}_d$. Due to Theorem 
\ref{main}, the fact that $(u'[h], U'[h] = (\tilde{w}[h], W[h])$, the limit 
\eqref{wtiljtowtil} and the variational formulation \eqref{perus} corresponding to the current pattern $\tilde{I}$, it holds that 
\begin{equation*}
\sum_{m=1}^M \tilde{I}_m(R'(h,I))_m = \lim_{j\to\infty}\sum_{m=1}^M\tilde{I}_m W_m^{(j)}[h] = \lim_{j\to\infty} B\big\{ (\tilde{w}^{(j)}[h],W^{(j)}[h]),(\hspace{1pt}\overline{\tilde{u}},\,\overline{\tilde{U}}) \big\}.
\end{equation*}
On the other hand, following the same line of reasoning as in \eqref{varformula} --- and approximating $\tilde{u}$ by a sequence of smooth functions $\{ \varphi_j \}$ --- we obtain that 
\begin{align*}
\lim_{j\to\infty} B\big\{ (\tilde{w}^{(j)}[h],W^{(j)}[h]),(\hspace{1pt}\overline{\tilde{u}},\,\overline{\tilde{U}}) \big\} = &-\int_{\partial\Omega} h_\nu(\sigma\nabla u)_\tau \cdot (\nabla\tilde{u})_\tau dS \nonumber \\
&+ \sum_{m=1}^M \frac{1}{z_m}\int_{E_m}h_\nu \frac{\partial u}{\partial\nu}(\tilde{U}_m - \tilde{u}) \, dS  \\
&- \sum_{m=1}^M \frac{n-1}{z_m}\int_{E_m} h_\nu H (U_m-u)(\tilde{U}_m-\tilde{u})\, dS, \nn
\end{align*}
which is the normal bundle version of (\ref{sampling}) and thus 
completes the proof. \hfill $\Box$

\section{Algorithmic implementation}
\label{sec:Algorithm}
In this section, we introduce our numerical algorithm for the simultaneous 
reconstruction of the admittivity distribution and the object boundary.
It is assumed that the object of interest $\Omega \subset {\R}^3$ is a cylinder 
$D \times (0,h_0)$, where $D \subset {\R}^2$ is a simply 
connected and bounded cross-section shape, and $h_0 > 0$ the known height of the body. The
electrodes are of the form $E_m = \gamma_m \times (0,h_0)$, with each
$\gamma_m$ being a connected part of $\partial D$ with a known length, i.e., 
the electrodes are rectangular, homogeneous in the vertical direction and assumed to be of a known width and the same height as the 
object itself. In particular, if the admittivity distribution were also 
homogeneous in the vertical direction --- as it is in our numerical experiments ---, the measurement setting could be 
modelled by a two-dimensional forward problem. Be that as it may, we
carry out all numerical computations in three dimensions in order to 
demonstrate the feasibility of our method in a realistic 
framework. Moreover, we only consider real-valued and isotropic electrical 
admittivities, i.e.,
$\sigma: \Omega \rightarrow \R_+$. Note that  the above assumptions on the target are made only for the sake of simplicity; the generalization of the algorithm
to more general three-dimensional settings is conceptually straightforward.

In the following three sections we outline the ideas behind our reconstruction method, but do not discuss all details about, e.g., the form of the smoothness prior for the admittivity; see, e.g.,~\cite{Kaipio00,Kaipio05} for more information.  

\subsection{Parametrization of the unknowns}
\label{sec:parameters}

In many practical situations the examined body has a star-shaped cross-section. In consequence, we search for the unknown boundary $\partial D$ as a $C^\infty$-curve parametrized by
 \begin{equation}\label{parametrization1}
{\gamma_\alpha}(\phi) = \bigg[\alpha_0 + \sum_{j=1}^N(\alpha_j \cos j\phi + \alpha_{j+N}\sin j\phi) \bigg] \begin{bmatrix}\cos\phi\\ \sin \phi\end{bmatrix}, \qquad \alpha_0, \ldots,\alpha_{2N} \in\R,
\end{equation}
where $\phi$ is the polar angle and the coefficients $\alpha = [\alpha_0, \dots, \alpha_{2N}]^{\rm T} \in \R^{2N +1}$ are assumed to be such that the curve does not intersect itself. Let $ D_\alpha $ denote the bounded set of $ \R^2 $ with $ \partial D_\alpha = \gamma_\alpha([0, 2\pi])$ and furthermore define $ \Omega_\alpha = D_\alpha \times (0,h_0) $. As it is assumed that the width of the (rectangular) electrodes is known, we may thus parametrize them by their initial polar angles $\theta_m$, $m=1, \dots, M$, in the counterclockwise direction. The vector containing these angles is denoted by $ \theta = [\theta_1,\ldots,\theta_M]^{\rm T} $. We assume that the electrodes are numbered in the natural order, that is, the terminal angle of an electrode precedes the initial angle of the following one. 

Approximate forward solutions to \eqref{cemeqs} in $ \Omega_\alpha $ are computed by a {\em finite element method} (FEM).
The FEM solver used in this work is an adaptation of the implementation in
\cite{Vauhkonen99}. In the FE scheme, we discretize
the computational domain~$\Omega_{\alpha}$ into tetrahedrons and approximate the distributions of admittivity and potential in piecewise linear and quadratic 
bases, respectively. In our reconstruction algorithm, the geometric parameters $ \alpha $ and $ \theta $ change iteratively and consequently $\Omega_\alpha$ and its FEM mesh also change at each step. In order to fix the admittivity discretization independently of such deformations, we pick a sufficiently large cylinder $ \Sigma = B\times (0,h_0) $, with a discoidal base $B$ inside which we let the cross-section $ D_\alpha $ evolve. Given this background cylinder, we look for admittivities of the form $ \sum_k\sigma_k\varphi_k $, where $ \{\sigma_k\}\subset (0,\infty) $ and $ \{\varphi_k\} $ is the piecewise linear basis related to a fixed tetrahedral mesh of $\Sigma $. The admittivity values are transformed between the fixed `reconstruction mesh' of $\Sigma$ and the varying ones in $\Omega_\alpha$ via linear interpolation. 

\subsection{Bayesian framework} 
Although our reconstruction algorithm, which will be introduced in Section~\ref{sec:algorithm}, cannot be considered purely Bayesian, its underlying motivation is statistical, and thus the basic ideas behind Bayesian inversion are outlined in the following.
In the Bayesian approach, all quantities are considered as random variables with some assumed prior probability distributions. Combining the information from the prior with the measurement data, one gets the updated posterior distribution for the parameters of interest \cite{Kaipio05}.

Let $ \{I^{(j)}\}_{j=1}^{M-1}$ be a basis of $ \R^ M_\diamond $. The voltages measured at the electrodes on the boundary of $\Omega_\alpha $ are modelled as 
\begin{equation}
\label{noisy}
V^{(j)}\, = \, U^{(j)}(\sigma,\alpha,\theta) + \eta^{(j)} \in \R^M,\qquad j=1,\ldots,M-1.
\end{equation}
Here, $(u^{(j)}(\sigma,\alpha,\theta),U^{(j)}(\sigma,\alpha,\theta))\in H^1(\Omega_\alpha)\oplus \R_\diamond^M$ is the solution of \eqref{cemeqs} in the domain~$ \Omega_\alpha $ with a real-valued admittivity $ \sigma $, when the net electrode current pattern $ I^{(j)} $ is injected through the $ M $ electrodes parametrized by their initial polar angles $ \theta \in \R^M $. Notice that we have fixed the ground level of potential by requiring that $ U^{(j)}(\sigma,\alpha,\theta) $ has vanishing mean. The components of the noise vector $ \eta^{(j)} \in \R^M $ are assumed to be independent realizations of zero mean Gaussian random variables. To simplify the notation, we pile the electrode currents, potentials and noise vectors into arrays of length $ M^2 - M $, i.e., employ the shorthand notations 
\begin{equation}\label{pile}\begin{split}
\mathcal{I} & = [(I^{(1)})^{\rm T},\ldots,(I^{(M-1)})^{\rm T}]^{\rm T},\\  \mathcal{V} & = [(V^{(1)})^{\rm T},\ldots,(V^{(M-1)})^{\rm T}]^{\rm T}\\
\eta & = [(\eta^{(1)})^{\rm T},\ldots,(\eta^{(M-1)})^{\rm T}]^{\rm T}\\ 
\mathcal{U}(\sigma,\alpha,\theta;\mathcal{I}) & = [U^{(1)}(\sigma,\alpha,\theta)^{\rm T},\ldots,U^{(M-1)}(\sigma,\alpha,\theta)^{\rm T}]^{\rm T}. 
\end{split}\end{equation}
This allows us to write the noisy measurement model as 
\begin{equation}\label{noisy_piled} 
\mathcal{V} = \, \mathcal{U}(\sigma,\alpha,\theta;\mathcal{I}) + \eta \in \R^{M^2-M}.
\end{equation}

The discretized admittivity is given a homogeneous Gaussian smoothness prior with a covariance $ \Gamma_\sigma$ and a positive homogeneous mean $ \sigma^\star $; for more details about smoothness priors, see \cite{Kaipio05}. To include control over the geometric information, the coefficients $\alpha$ are provided with a Gaussian prior with a mean $ \alpha^\star \in \R^{2N+1} $ and a covariance matrix $ \Gamma_\alpha = {\rm diag}(a_0^2,\ldots,a_{2N}^2), $ where 
\begin{equation}\label{alpha_prior}
a_j = \begin{cases} l^{-s}a, & \; l=0,\ldots,N \\  (l-N)^{-s}a,& \; l=N+1,\ldots, 2N. \end{cases}
\end{equation}
By adjusting the parameters $ s,a>0$, one may tune the prior assumption on the regularity of the object boundary. The electrode initial polar angles are also given a Gaussian prior density with a mean $ \theta^\star \in \R^M $ and a diagonal covariance matrix $ \Gamma_\theta = \tau^2 \mathbb{I} $, where $\tau > 0$ is the corresponding standard deviation. As a result, for a measurement $ \mathcal{V} $ of the form \eqref{noisy}, a {\em maximum a posteriori} (MAP) estimate is obtained as a minimizer of the Tikhonov-type functional 
\begin{align}
\label{mapfunctional}
\Phi(\sigma,\alpha,\theta) \, := &\, \,  (\mathcal{U}(\sigma,\alpha,\theta;\mathcal{I}) -\mathcal{V} )^{\rm T} \Gamma_\eta^{-1}(\mathcal{U}(\sigma,\alpha,\theta;\mathcal{I}) -\mathcal{V} )   + (\sigma-\sigma^\star)^{\rm T} \Gamma_\sigma^{-1}(\sigma-\sigma^\star) \nonumber \\[1pt] &+ \, (\alpha -\alpha^\star)^{\rm T} \Gamma_\alpha^{-1}(\alpha-\alpha^\star) +\, \tau^{-2}|\theta - \theta^\star|^2 ,
\end{align}
where $ \Gamma_\eta $ is the diagonal noise covariance matrix; see \cite{Kaipio05}. 

\subsection{The (quasi-Bayesian) iterative algorithm} 
\label{sec:algorithm}
According to our experience, the EIT measurements modelled by the CEM are 
typically more sensitive to the exterior boundary shape and electrode 
locations than to the internal admittivity distribution. As a consequence, 
it seems to be computationally advantageous to first fix a crude 
constant approximation for the admittivity distribution, then use a (deterministic) 
iterative scheme to come up with a relatively good model for the object 
boundary and electrode positions, and finally use these preliminary estimates 
as the prior expectations in the to-be-minimized MAP functional 
\eqref{mapfunctional}. It should be emphasized that such an initialization
of the means makes our algorithm strictly speaking non-Bayesian, since the choice of the 
priors should be independent of the data. Be that as it may, according to our experience, such a preliminary step leads to faster and more reliable convergence.
(We do not claim, however, that this kind of two-step implementation is
the only feasible choice.)

To be more precise, we first choose the covariance matrices $\Gamma_\eta$, 
$\Gamma_\sigma$, $\Gamma_\alpha$ and $\Gamma_\theta$ according to the assumed prior 
information on the variation of the corresponding parameters, 
pick an initial guess 
$(\alpha^{(0)}, \theta^{(0)})$ for the measurement geometry (corresponding to some disk-shaped cross-section in all of our numerical studies), 
and fix~$\sigma^\star$ to be the {\em constant} admittivity that 
minimizes the output least squares part of $\Phi$, i.e., 
\begin{equation*}
(\mathcal{U}(\sigma,\alpha,\theta;\mathcal{I}) -\mathcal{V} )^{\rm T} \Gamma_\eta^{-1}(\mathcal{U}(\sigma,\alpha,\theta;\mathcal{I}) -\mathcal{V} )
\end{equation*}
when $(\alpha, \theta) = (\alpha^{(0)}, \theta^{(0)})$.
Subsequently, the following two-stage scheme is employed:

\subsubsection*{First stage of the algorithm: choosing the prior means}
We apply a Levenberg--Marquardt type method in order to 
choose the geometry
parameters that are used as the prior means $(\alpha^\star,\theta^\star)$ 
when $\Phi$ of \eqref{mapfunctional} is minimized simultaneously with respect 
to all of its variables in the second stage of the algorithm: 
\begin{enumerate}
\item Fix $\sigma =
\sigma^\star$ in \eqref{mapfunctional}, consider $\Phi$ as a function
of only two  variables $\alpha$ and $\theta$, and set 
$(\alpha^\star, \theta^\star) = (\alpha^{(0)}, \theta^{(0)})$.
\item Calculate the Gauss--Newton minimization direction for $\Phi(\alpha, \theta)$ of 
\eqref{mapfunctional} at $(\alpha, \theta) = (\alpha^\star, \theta^\star)$; see, e.g., \cite{Nocedal99}.
\item Minimize $\Phi(\alpha, \theta)$ over the line passing through $(\alpha^\star, \theta^\star)$ in the Gauss--Newton direction by the Golden section line 
search. Redefine $(\alpha^\star, \theta^\star)$ to be the obtained minimizer.
\item If satisfactory convergence is achieved, terminate the iteration.
Otherwise, return to step $2$.
\end{enumerate}
This part of the reconstruction algorithm is stable and does not, in particular,
seem very sensitive with respect to the choice of the covariance matrices in 
\eqref{mapfunctional}.

\subsubsection*{Second stage of the algorithm: finding the MAP estimate}
After the prior means $(\sigma^\star, \alpha^\star, \theta^\star)$ have
been chosen in the earlier parts of the algorithm, the final stage
consists of minimizing $\Phi$ of \eqref{mapfunctional} by the Gauss--Newton
algorithm:
\begin{enumerate}
\item Set $k=1$ and $(\sigma^{(k)}, \alpha^{(k)}, \theta^{(k)}) =(\sigma^\star, \alpha^\star, \theta^\star)$.
\item Calculate the Gauss--Newton direction for $\Phi{(\sigma, \alpha,\theta)}$ of \eqref{mapfunctional}
at $(\sigma, \alpha, \theta) = (\sigma^{(k)}, \alpha^{(k)}, \theta^{(k)})$. 
\item Minimize $\Phi$ over the line passing through $(\sigma^{(k)}, \alpha^{(k)}, \theta^{(k)})$ in the Gauss--Newton direction by the Golden section line 
search and define $(\sigma^{(k+1)}, \alpha^{(k+1)}, \theta^{(k+1)})$ to be the obtained minimizer.
\item Unless  satisfactory convergence is achieved, increase $k$ by one
and return to step $2$.
\end{enumerate}

\vspace{2mm}

For the computation of the needed Gauss--Newton directions, one needs the Jacobian of $ \mathcal{U}(\sigma,\alpha,\theta;\mathcal{I}) $ with respect to $ \sigma $, $ \alpha $ and $ \theta $. By the Jacobian with respect to $ \sigma $ we mean the one with respect to the coefficients of the piecewise linear basis in the `reconstruction cylinder' $\Sigma$ introduced in the last paragraph of Subsection~5.1. (Notice that the coefficients of the basis functions supported outside $\Omega_\alpha$ do not play a role in $ \mathcal{U}(\sigma,\alpha,\theta;\mathcal{I}) $, but they {\em do} affect the last term on the first line of \eqref{mapfunctional}.) For the estimation of the derivatives with respect to $\sigma$ and $\theta$, we refer to \cite{Kaipio00} and \cite{Darde12}, respectively. By the dual relation \eqref{sampling}, the Jacobian with respect to $ \alpha $ can be sampled via trivial linear algebra (a change of basis) after evaluating the expressions on the right-hand side of \eqref{sampling} for each triplet 
$$ 
h(\phi)=  \psi^{(l)}(\phi) \begin{bmatrix} \cos \phi \\ \sin \phi
\end{bmatrix}, \quad (u,U) = (u^{(i)},U^{(i)}), \quad 
(\tilde{u},\tilde{U}) = (u^{(j)},U^{(j)}) 
$$ 
over the indices $ l = 0,\ldots, 2N $ and $ i,j = 1,\ldots, M-1 $. Here,
$(u^{(j)},U^{(j)}) = (u^{(j)}(\sigma,\alpha,\theta),U^{(j)}(\sigma,\alpha,\theta)) $ is the solution of \eqref{cemeqs} for $I = I^{(j)}$ and the setting parametrized by
$(\sigma,\alpha,\theta)$, and
$ \psi_l(\phi) = \cos l\phi $ if $ l\leq N $ and $ \psi_l(\phi)= \sin (l-N)\phi $ when $ l \geq N+1 $. We emphasize that one needs not solve any extra forward problems for this procedure since at each iteration step all the pairs $ (u^{(j)}(\sigma,\alpha,\theta),U^{(j)}(\sigma,\alpha,\theta)) $, $ j = 1,\ldots, M-1 $, must be computed already for evaluating the functional $ \Phi $ of \eqref{mapfunctional}. By the assumption that the electrode width is known, for any given electrode the terminal polar angle is a smooth function of the initial one and $ \alpha $. This functional dependence can be written explicitly by employing the arc length formula for the parametrization \eqref{parametrization1}, and this information can then be included in the Jacobians with the help of the Leibniz rule and the chain rule for the total derivative.

\begin{remark}
If one chooses to skip the first stage of the above introduced algorithm and
use the (simple) initial guess $(\alpha^{(0)},\theta^{(0)})$ as the prior 
mean for the geometric parameters in the second stage, 
with suitable parameter choices the reconstructions for the numerical experiments of the
following section typically remain qualitatively the same, but the convergence 
slows down considerably. What is more, in practice the initial guess  
$(\alpha^{(0)},\theta^{(0)})$ for the measurement geometry is often more 
accurate than the ones we employ in our numerical experiments, which further
reduces the practical relevance of the first stage of the algorithm.
\end{remark}

\section{Numerical experiments}
\label{sec:Numerics}
Our main aim is not so much to compare the functionality of our method with reconstruction techniques presented elsewhere, but to make an `internal' comparison between three cases: 
\begin{itemize}
\item[(i)] the measurement geometry, i.e.~the object shape and the electrode locations, is known; 
\item[(ii)] the measurement geometry is known inaccurately but this is not taken into account in the algorithm; 
\item[(iii)] the unknown boundary shape is estimated simultaneously with the admittivity distribution. 
\end{itemize}
We will demonstrate that (i) and (iii) give comparable results, while the quality of reconstructions for (ii) is intolerably bad.

We present three numerical experiments, in each of which $M=16$ identical electrodes of known width are attached to the object of interest. We assume to know the contact impedances and choose the values $z_m=1$, $m=1, \dots, M$. 
The first experiment, though a bit impractical, acts as an initial probe to test the functionality of the computed Fr\'echet derivatives: we apply (the first part of) our algorithm to the shape estimation of a target object with a known homogeneous admittivity distribution. In the second experiment we consider a simple shape (an ellipse) and a smooth admittivity distribution. 
In the last experiment the object shape is moderately complicated and the admittivity phantom consists of inclusions of constant admittivity in a homogeneous background. 

Let $ \varsigma $ be the to-be-reconstructed admittivity and suppose the pair $(\beta, \vartheta)$ provides a parametrization of the target measurement setting in the sense of Section~\ref{sec:parameters}. To be quite precise, the latter statement is a bit ambiguous because {\em none} of the considered target shapes $\partial D$ can be given in the form \eqref{parametrization1} with a finite $N$, but for ease of notation we have decided to allow here an `infinite' shape parameter vector $\beta$.
For each experiment we simulate the exact data $ \mathcal{U}(\varsigma,\beta,\vartheta;\mathcal{I}) $ using the input current basis $ I^{(j)} = {\rm e}_1 - {\rm e}_j \in\R^M_\diamond$, $ j = 2,\ldots,M $, where  $ {\rm e}_j $ is the $ j $th Euclidean basis vector.
Notice that there is no danger of an inverse crime because a new finite element mesh for the approximate domain $ \Omega_{\alpha} $ is generated at each iteration of the reconstruction algorithm, and these meshes differ considerably from the mesh of the target object $\Omega = \Omega_\beta$ used for the data simulation.

The actual noisy measurement realization $ \mathcal{V} $ is formed via \eqref{noisy_piled}, with $(\sigma,\alpha,\theta) = (\varsigma,\beta,\vartheta)$, by picking a particular noise component $ \eta_m^{(j)} $, $ m=1,\ldots,M $, $ j=1,\ldots,M-1 $, from a zero mean Gaussian distribution with the variance 
\begin{equation}\label{sim_noise}
0.01^2|U_m^{(j)}(\varsigma,\beta,\vartheta)|^2 + 0.001^2\max_{1 \leq k,l \leq M} | U_k^{(j)}(\varsigma,\beta,\vartheta) - U_l^{(j)}(\varsigma,\beta,\vartheta) |^2. 
\end{equation}
Here, the relation between $ \mathcal{U}(\varsigma,\beta,\vartheta;\mathcal{I}) $ and $ U_m^{(j)}(\varsigma,\beta,\vartheta) $ is as in \eqref{pile}. 
Sections~\ref{sec:smooth}  and \ref{sec:piecewise}, where the cases 
(i--iii) are compared, work with fixed realizations of the noise vector $\eta$ 
in order to allow a fair comparison.

For further justification of the noise model \eqref{sim_noise}, see \cite{Hyvonen10}, but anyway notice that \eqref{sim_noise} corresponds to more than one percent of relative noise in the {\em absolute} data, which is a substantial amount for an EIT problem. In each numerical experiment we assume to know the covariance of the measurement noise, i.e., we use the diagonal covariance matrix defined by the noise model \eqref{sim_noise} as $ \Gamma_\eta $ in \eqref{mapfunctional}. We do not elaborate on the choice of $ \Gamma_\sigma $ in \eqref{mapfunctional} in further detail; it is built based on the proper (informative) smoothness prior proposed in \cite{Kaipio05}, reflecting the {\em a priori} assumption on the spatial variations of the admittivity.


\subsection{Known homogeneous target admittivity} 
Figure \ref{example1} shows the results obtained when the {\em first stage} of our algorithm is applied to reconstructing the boundary shape and electrode locations for an object with a known homogeneous admittivity distribution $\varsigma\equiv 1$. While this situation has minor practical relevance, it serves as a test of the computational techniques for obtaining the derivatives with respect
to $\alpha$ and $\theta$. The cylindrical target object is $ \Omega = D\times (0,h_0) $ with $ h_0=1 $, and the curve $ \partial D $ is parametrized by 
\begin{equation*}
{\gamma}(\phi) = \bigg[ \frac{3.3}{(2.2^2\cos^2\phi + 1.5^2\sin^2\phi)^{1/2}} + 1.1e^{-(\phi-\pi)^6} + 0.88\cos\phi\sin(-2\phi) \bigg]\begin{bmatrix}\cos\phi\\ \sin\phi\end{bmatrix}.
\end{equation*}
The width of the $ M=16 $ identical electrodes on $ \partial D\times (0,h_0) $ is $ 0.3 $ and their initial polar angles are of the form $ \vartheta_m = 2\pi(m-1)/M + \varepsilon_m $, $ m=1,\ldots,M $, where $ \varepsilon_m $ are independent realizations of a normally distributed random variable with zero mean and standard deviation $0.1$.

\begin{figure}[t!]
	\begin{center}
        \subfloat[The iterates.]{\includegraphics[width=0.45\textwidth]{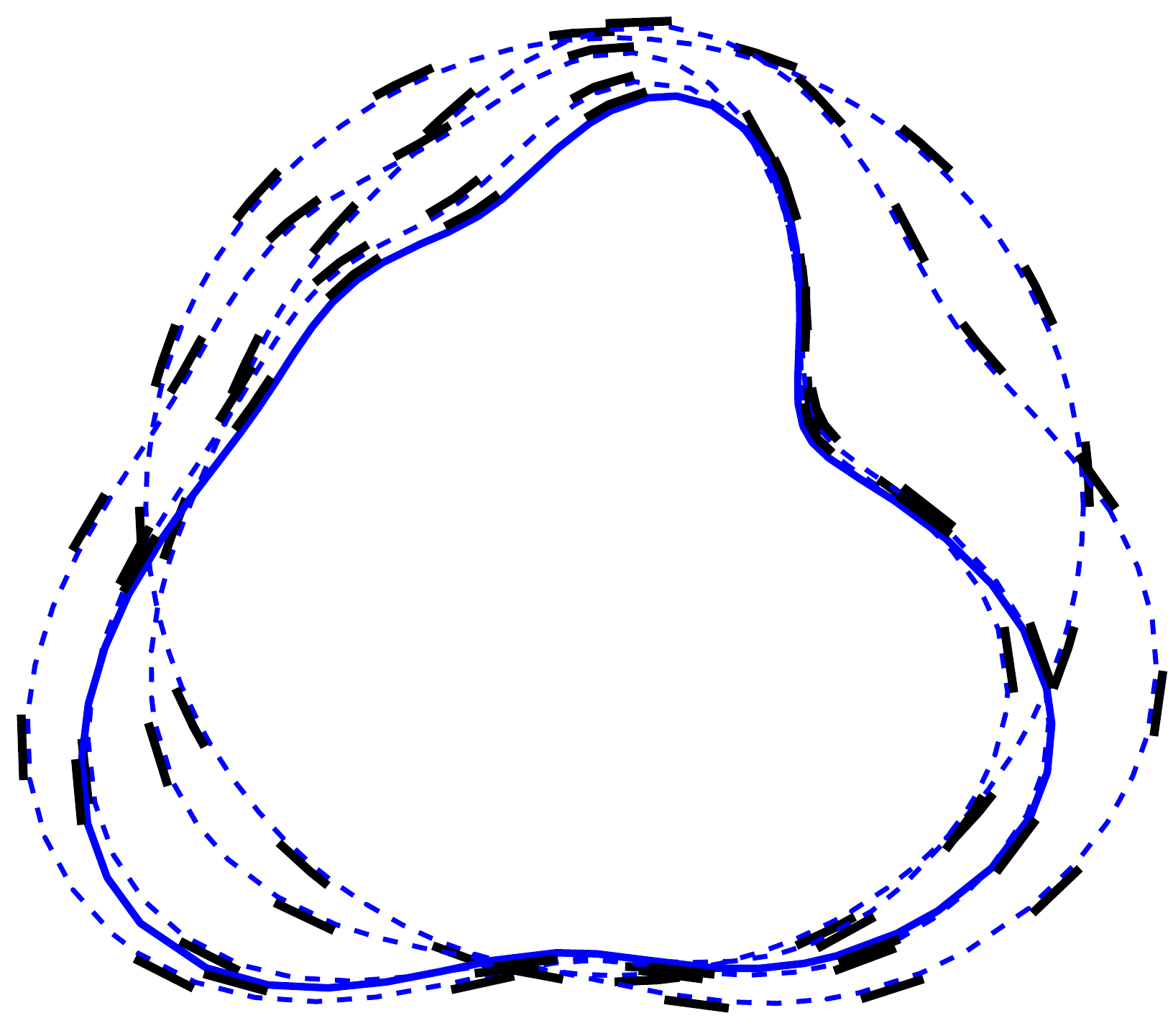}} \qquad \subfloat[Reconstructed geometry.]{\includegraphics[width=0.4\textwidth]{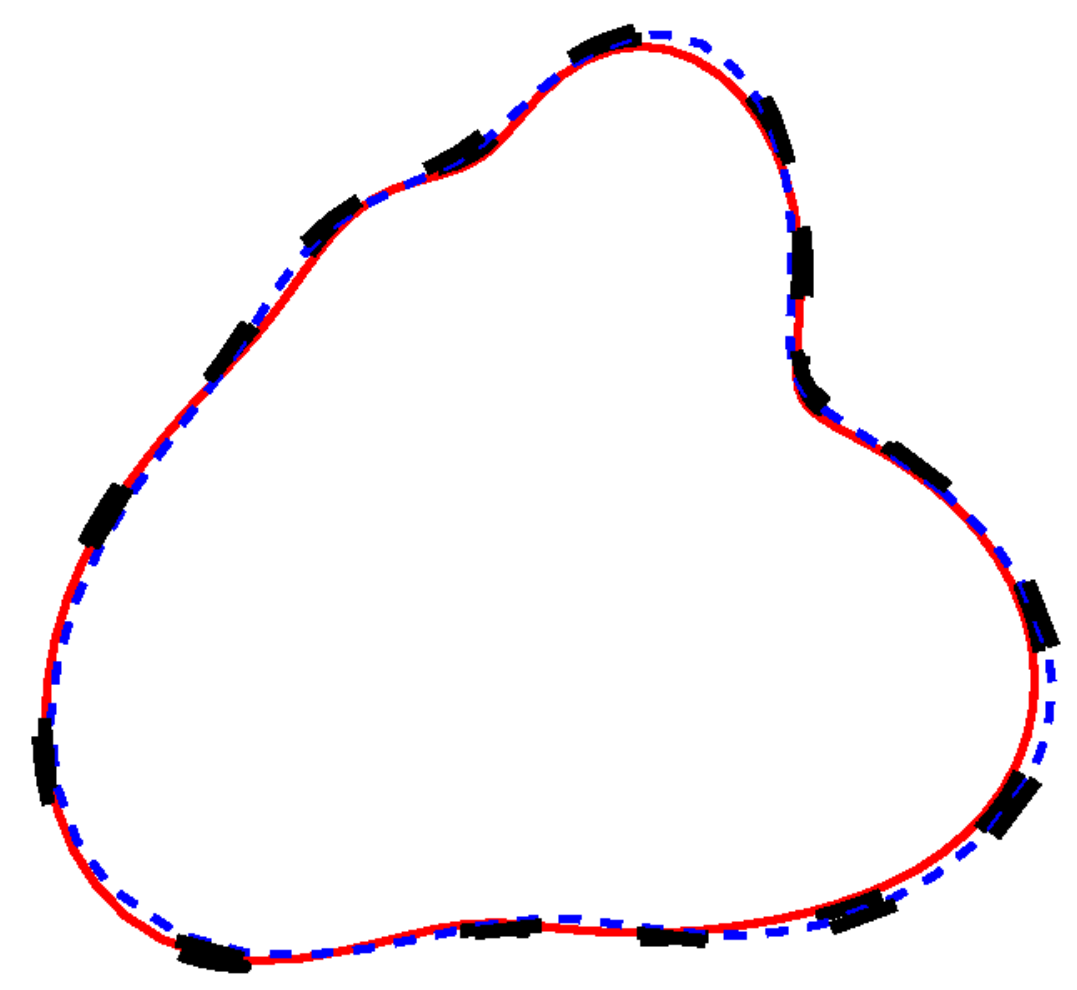}}\\
	\end{center}
	\caption{Retrieval of an unknown cross-section shape from noisy simulated data. (a)~The five iterates with the final one plotted with solid line. (b)~Comparison between the exact shape (red solid) and the retrieved one (blue dashed).}
	\label{example1}
\end{figure} 


Since the admittivity is known {\em a priori}, we do not estimate $\sigma^\star$ as explained in the beginning of Section~\ref{sec:algorithm}, but fix it to be identically $1$.
The first stage of the algorithm in Section~\ref{sec:algorithm} is then 
run with $ N=15$ and $M=16 $; in this case the final value of $(\alpha^\star, \theta^\star)$ describes the reconstructed measurement geometry. In the construction of the prior covariance $ \Gamma_\alpha $ we use \eqref{alpha_prior} with the selection $ a = s = 1 $ and set $ \Gamma_\theta = \tau^2 \mathbb{I} $ with $ \tau = 2\pi/M $, but the algorithm does not seem to be very sensitive with respect to these choices. Figure \ref{example1}(a) shows the required five iteration steps, one of which is the eventual reconstruction, obtained by choosing the initial guesses $ \alpha^{(0)} = [2.7,0,\ldots,0]^{\rm T} $ and $ \theta_m^{(0)} = 2\pi(m-1)/M $, $ m=1,\ldots,M-1 $. The final iterate is drawn with solid line and the others with dashed line. In Figure \ref{example1}(b), the target curve $ \gamma $ (red solid) is compared with the retrieved one (blue dashed).

The algorithm was run with several different target objects and in 
all cases the results were qualitatively similar to what is illustrated in Figure~\ref{example1}, given that the examined shapes were not too complicated: If there were fine structures on a scale smaller than the electrode width, the results were poor. Further, the simpler the geometry, the faster the convergence was. It was also observed that the number of coefficients in \eqref{parametrization1} should not be too large, at most about $N=15$. A high number of coefficients results in unstable reconstructions and absurd shapes, with the performance of the algorithm slowing down.

\subsection{Smooth target admittivity} 
\label{sec:smooth}
In the second experiment, we apply the (whole) simultaneous reconstruction algorithm to data corresponding to a relatively simple target shape and a smooth admittivity distribution illustrated in Figure~\ref{example2}(a).   The shape of the target object is $ \Omega = D \times (0,h_0) $, $ h_0 = 0.5 $, where $ D $ is an ellipse with major and minor semi-axes $ 2 $ and  $ 1.5 $, respectively. 
The admittivity is homogeneous in the vertical direction, which allows us to only consider cross-sections in the visualizations.
The electrode positions are chosen in the same manner as in the previous example, with the constant electrode width being such that two fifths of $ \partial D\times (0,h_0) $ is covered by the electrodes. 

\begin{figure}[t!]
	\begin{center}
       \subfloat[Phantom.]{\includegraphics[width=0.45\textwidth]{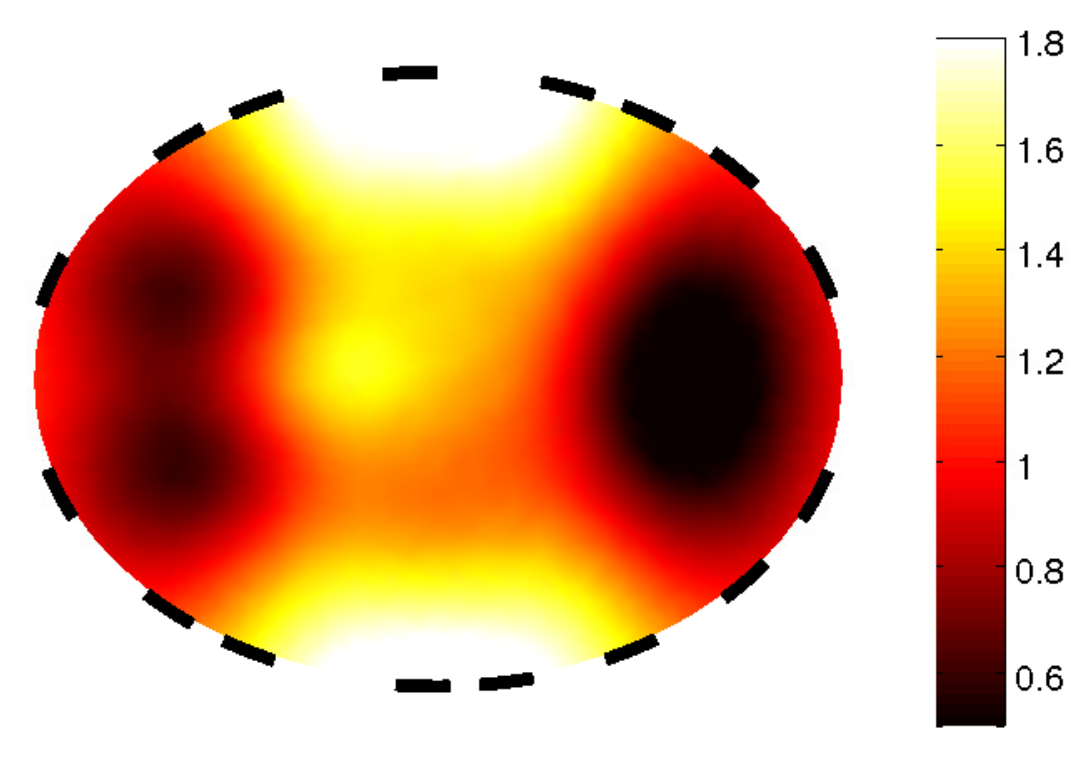}}\qquad  \subfloat[Incorrect geometry.]{\includegraphics[width=0.45\textwidth]{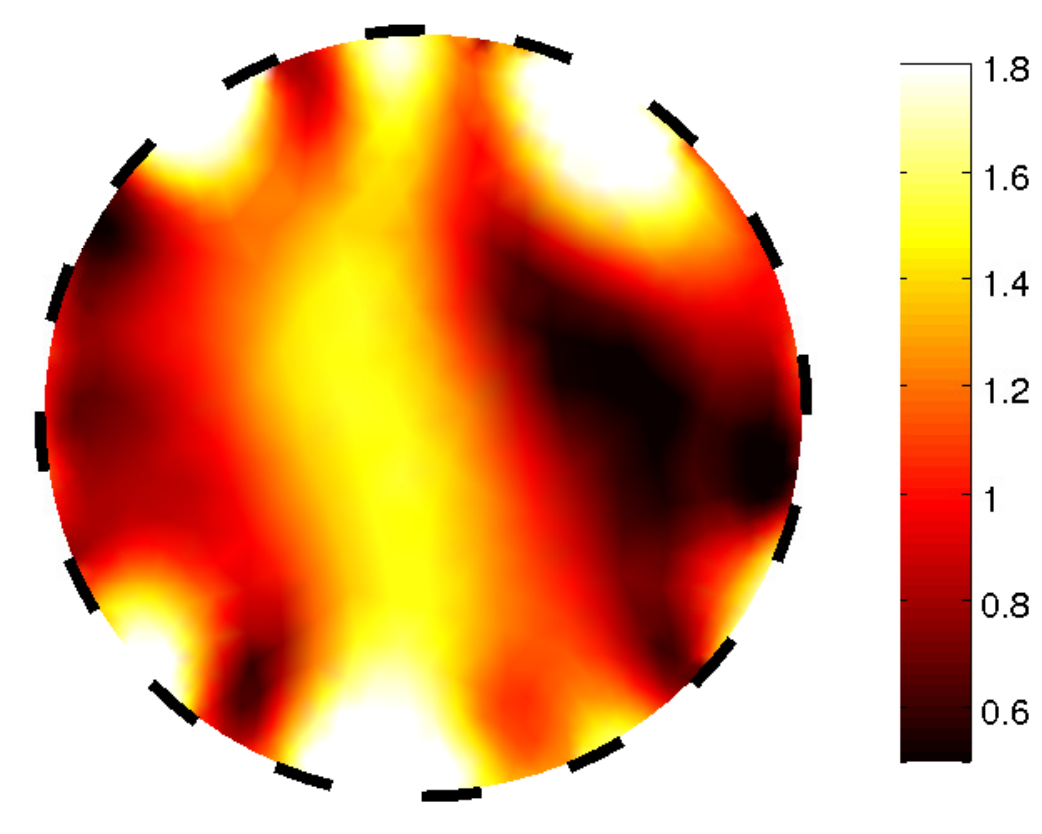}}\\[2mm]
        \subfloat[Correct geometry.]{\includegraphics[width=0.45\textwidth]{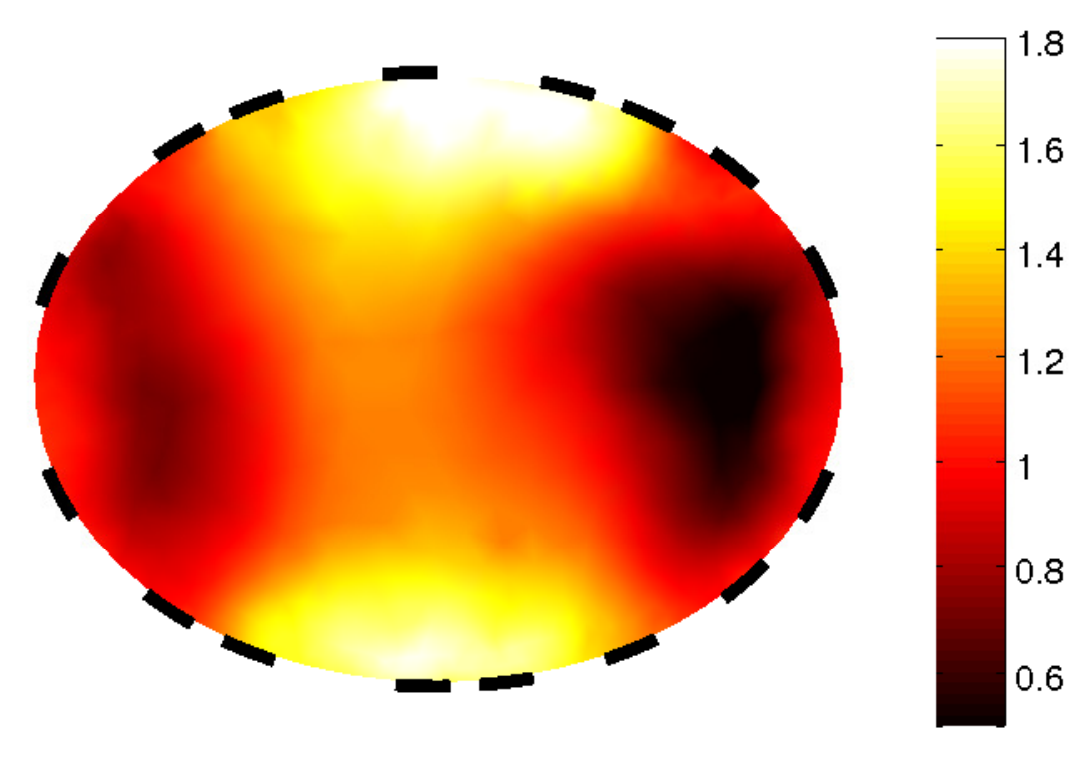}} \qquad  \subfloat[Simultaneous retrieval.]{\includegraphics[width=0.45\textwidth]{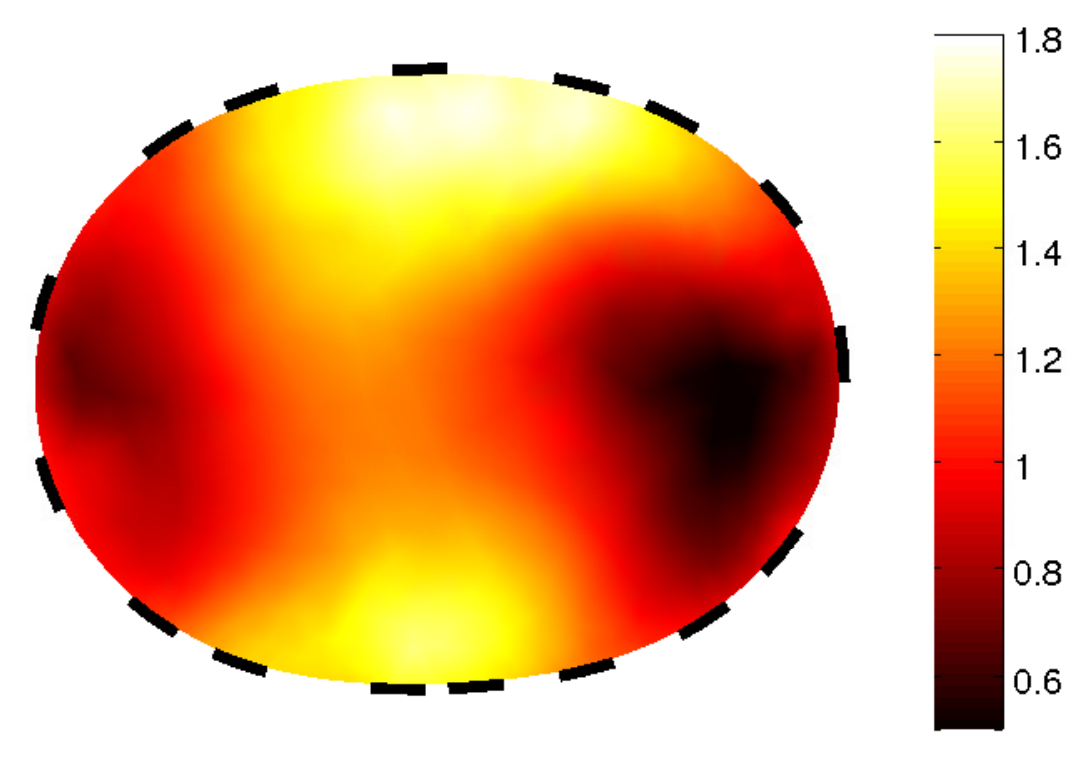}}\\
	\end{center}
	\caption{Experiment with a simple target shape and a smooth admittivity; each image represents a cross-section of the corresponding
phantom/reconstruction that is (almost) homogeneous in the vertical direction. (a)~Phantom used in data simulation. (b) Reconstruction corresponding to an incorrect fixed geometry. (c) Reconstruction corresponding to the exact geometry. (d) Simultaneously reconstructed  admittivity and measurement geometry.
}
\label{example2}
\end{figure}

In the reconstruction process we seek for a parameter triplet $ (\sigma,\alpha,\theta)\in \R^K\times \R^N\times \R^M $ with $ N=7$ and $M=16 $. Here, $ K $ is the number of nodes in the (fine enough) discretization of the background cylinder $ \Sigma = B \times (0,h_0)$, with $B$ chosen to be an origin-centered disk of radius $3$. As the initial guesses, we use 
$ \alpha^{(0)} = [2,0,\ldots,0]^{\rm T} $ and $ \theta_m^{(0)} = 2\pi(m-1)/M $, $ m=1,\ldots,M-1 $. The prior covariances are constructed by selecting $ a = 0.1 $, $ s = 1 $ for $ \Gamma_\alpha $ of \eqref{alpha_prior} and $ \Gamma_\theta = \tau^2 \mathbb{I} $ with $ \tau = 2\pi/M $; see the fourth paragraph of Section~\ref{sec:Numerics} for an explanation about the choice of
$\Gamma_\sigma$ and~$\Gamma_\eta$.

The reconstruction in Figure \ref{example2}(b) was obtained by applying the {\em second stage} of the algorithm in Section~\ref{sec:algorithm} with respect to $\sigma$ to the setting where the last two terms of \eqref{mapfunctional} are deleted and $(\alpha,\theta)$ is fixed to be the initial guess $ (\alpha^{(0)},\theta^{(0)}) $; this approach corresponds to ignoring the incompleteness of the information on the measurement configuration and assuming stubbornly that the cross-section of the target object is a disk with uniformly distributed electrodes on its boundary. The reconstruction corresponding to the precise knowledge of the geometry is depicted in Figure~\ref{example2}(c); it was obtained in the same manner as the one in Figure~\ref{example2}(b), except this time around the geometry parameters $(\alpha, \theta) = (\beta, \vartheta)$ were fixed at the values describing the target configuration used in the simulation of the measurement data. 
Figure~\ref{example2}(d) visualizes simultaneous retrieval of the admittivity distribution and the measurement geometry; this reconstruction corresponds to application of the {\em whole} two-stage reconstruction algorithm of Section~\ref{sec:algorithm}, starting from the initial guess $(\alpha^{(0)},\theta^{(0)}) $ defined above. 

From Figure~\ref{example2}(b), it is obvious that ignoring the incompleteness of the information on the measurement configuration results in severe artefacts in the admittivity reconstruction close to the object boundary. On the other hand, a comparison of Figure~\ref{example2}(d) with  Figures~\ref{example2}(c) and \ref{example2}(b) demonstrates that the simultaneous retrieval of the admittivity distribution and
the measurement setting provides a qualitatively similar reconstruction as knowing
the exact geometry to begin with, and a far better one than
altogether ignoring the inaccuracies in the geometric information.


\begin{figure}[b!]
	\begin{center}
       \subfloat[Phantom.]{\includegraphics[width=0.45\textwidth]{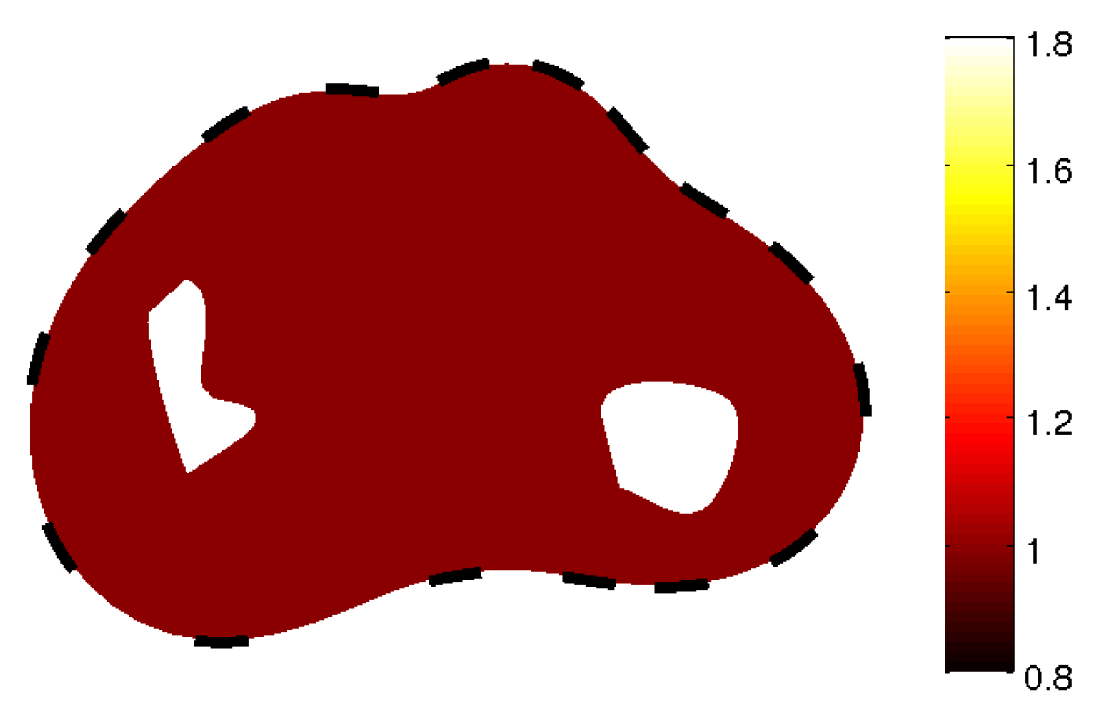}}\qquad  \subfloat[Incorrect geometry.]{\includegraphics[width=0.45\textwidth]{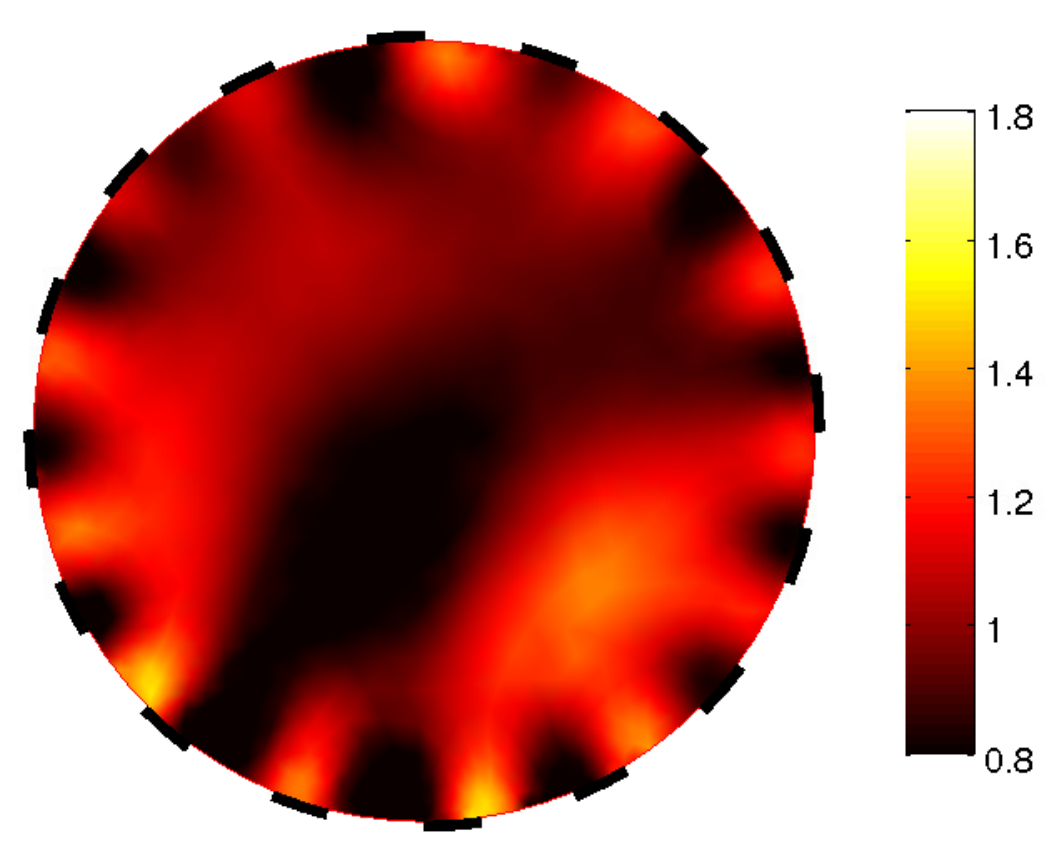}}\\[2mm]
               \subfloat[Correct geometry.]{\includegraphics[width=0.45\textwidth]{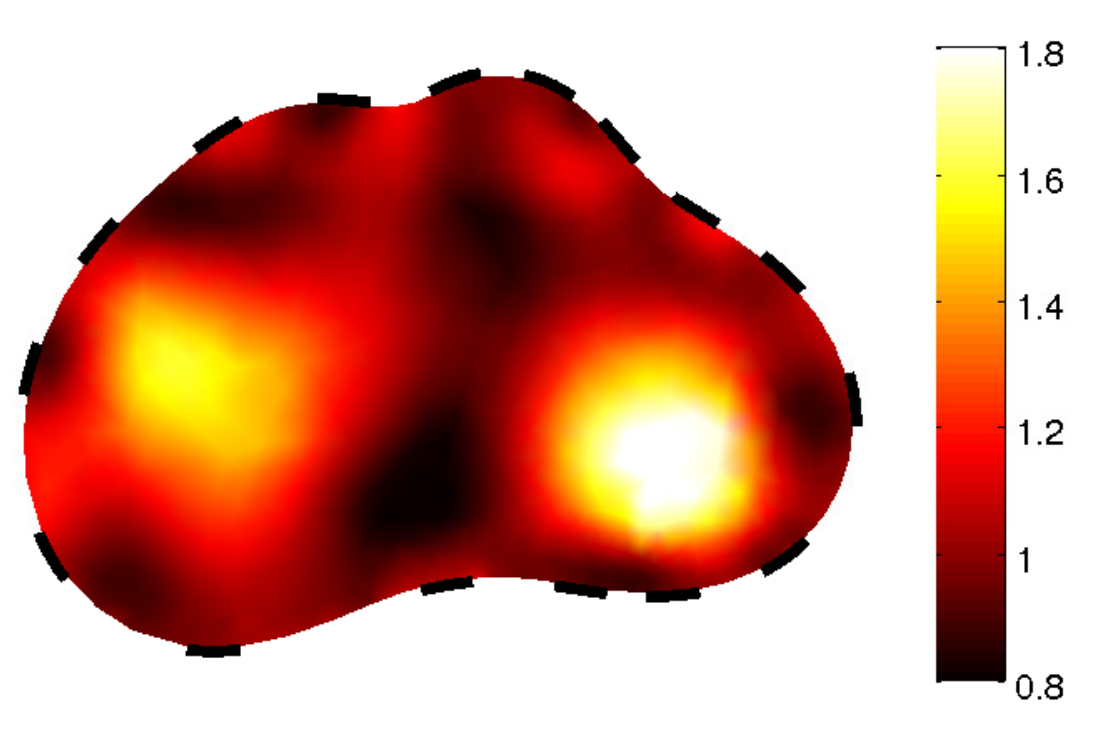}}\qquad  \subfloat[Simultaneous retrieval.]{\includegraphics[width=0.45\textwidth]{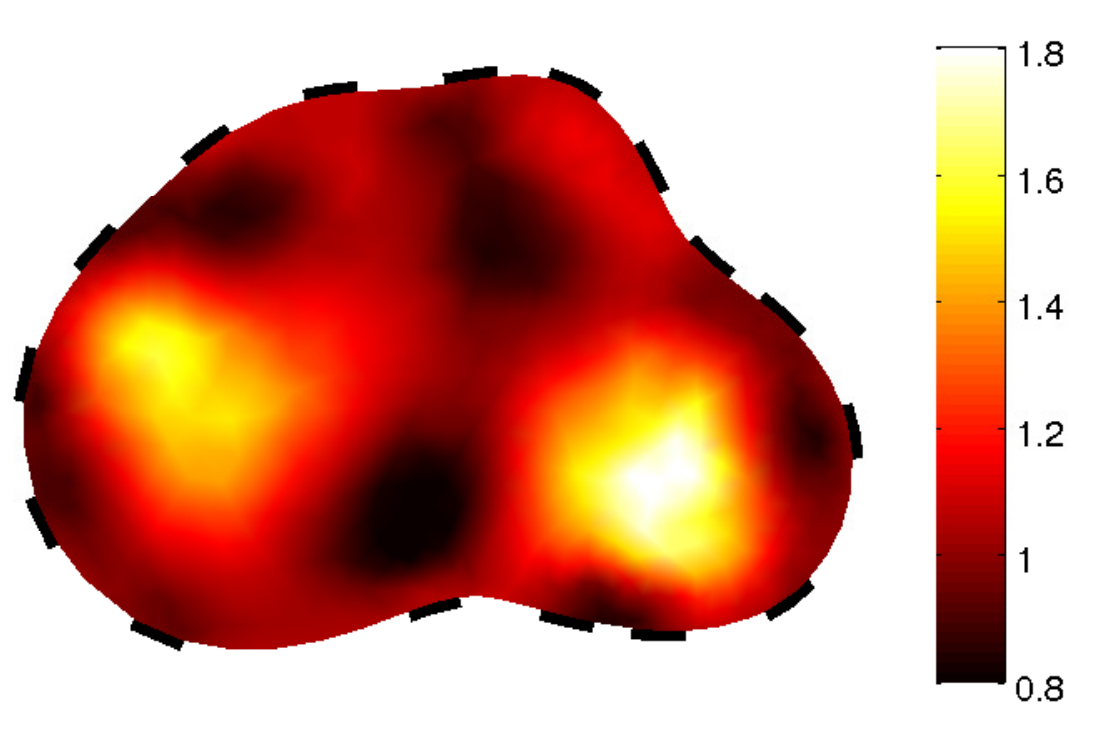}}\\
	\end{center}
	\caption{Experiment with a complicated target shape and a piecewise 
constant admittivity; each image represents a cross-section of the corresponding
phantom/reconstruction that is (almost) homogeneous in the vertical direction. (a)~Phantom used in data simulation; the admittivity of the inclusions is~$10$. (b) Reconstruction corresponding to an incorrect fixed geometry. (c) Reconstruction corresponding to the exact geometry. (d) Simultaneously reconstructed  admittivity and measurement geometry.
}
\label{example3}
\end{figure}

\subsection{Piecewise constant target admittivity} 
\label{sec:piecewise}
In our last experiment, we consider the target object illustrated in Figure \ref{example3}(a). It is characterized by $ \Omega = D\times (0,h_0) $, $ h_0 = 0.5 $, with $ \partial D $ parametrized by 
\begin{equation*}
{\gamma}(\phi) = \bigg[ \frac{3}{(1.5^2\cos^2\phi + 2^2\sin^2\phi)^{1/2}} + 0.75e^{-(\phi-\pi)^6} + 0.6\cos\phi\sin(-2\phi) \bigg]\begin{bmatrix}\cos\phi\\ \sin\phi\end{bmatrix}.
\end{equation*}
The corresponding admittivity distribution, which is homogeneous in the vertical direction, consists of a homogeneous unit background and two embedded inclusions with the constant admittivity level 10. 
The target electrodes are of equal width, they cover two fifths of $ \partial D $ and their locations are chosen as in the previous examples.

In this case we consider $ \Phi $ of \eqref{mapfunctional} as a function of $ (\sigma,\alpha,\theta)\in \R^K\times \R^N \times \R^M $, with $ N = 15 $, $ M=16 $ and $ K $ being the number of nodes in the mesh for the background cylinder $ \Sigma = B \times (0,h_0)$. Here, $B$ is once again a disk of radius $3$ centered at the origin. We assume the same prior information as in the previous example: $ \Gamma_\alpha $ is as in \eqref{alpha_prior} with $ a = 0.1$, $ s = 1 $ and $ \Gamma_\theta = \tau^2\mathbb{I} $ with $ \tau = 2\pi/M $. The initial guesses for the iterative reconstruction algorithm of Section~\ref{sec:algorithm} are 
$ \alpha^{(0)} = [1.5,0,\ldots,0]^{\rm T} $ and $ \theta_m^{(0)} = 2\pi(m-1)/M $, $ m=1,\ldots,M-1 $. 

The results are illustrated in Figure~\ref{example3}, with the
subimages organized as in Figure~\ref{example2} of the previous section. The reconstruction 
shown in Figure~\ref{example3}(b) was obtained by ignoring the incompleteness of
the information on the geometry, i.e., applying the {\em second stage} of the
reconstruction algorithm with respect to $\sigma$ when the second line
of \eqref{mapfunctional} is deleted and $(\alpha,\theta) = (\alpha^{(0)}, \theta^{(0)})$ is fixed. Figure~\ref{example3}(c)
corresponds to the precise knowledge of the measurement
setting, i.e., again ignoring the second line of \eqref{mapfunctional},
but fixing $(\alpha,\theta) =  (\beta, \vartheta)$ to be the parameter values describing the target 
configuration. Finally, the reconstruction in Figure~\ref{example3}(d) 
visualizes simultaneous retrieval of the admittivity distribution and the 
measurement geometry by the {\em whole} two-stage algorithm of Section~\ref{sec:algorithm}.

The conclusions about the functionality of the different approaches are the same as in the previous experiment: The simultaneous
retrieval of the admittivity distribution and the measurement geometry 
provides a reconstruction that is comparable to the case that the object shape
and electrode locations are known accurately. On the other hand, ignoring the uncertainties
in the measurement configuration gives a poor outcome. 



%

\section{Concluding remarks}
\label{sec:Conclusion}

We have presented the Fr\'echet derivative of the measurement map of practical EIT with respect to the (exterior) object boundary shape as a part of the solution to a certain elliptic boundary value problem. Through three-dimensional numerical studies based on simulated data, we have demonstrated that utilizing such a geometric derivative, the estimation of the object  shape 
and the electrode locations can be incorporated into a Newton-type output least squares reconstruction algorithm in the framework on the  CEM of EIT.

\end{document}

%% file: rateinit2.tex
\newcommand{\bthm}{\begin{thm}}
\newcommand{\ethm}{\end{thm}}
\newcommand{\bpr}{\begin{proof}}
\newcommand{\epr}{\end{proof}}

\newcommand{\sg}{\sigma}

\newcommand{\lb}{\left\lbrace}
\newcommand{\rb}{\right\rbrace}
\newcommand{\bra}[1]{\lb #1\rb}
\newcommand{\lve}{\left\lvert}
\newcommand{\rve}{\right\rvert}
\newcommand{\abs}[1]{\left\lvert#1\right\rvert}

\newcommand{\pa}[1]{\left(#1\right)}

\newcommand{\Set}[2]{\lb #1 \;\middle\vert\; #2 \rb}

\newcommand{\del}{\nabla}

\newcommand{\nder}{\nu\cdot\sg\del}

\newcommand{\R}{\mathbb{R}}
\newcommand{\C}{\mathbb{C}}
\newcommand{\N}{\mathbb{N}}
\newcommand{\Z}{\mathbb{Z}}

\newcommand{\F}{\mathcal{F}}

\newcommand{\B}{\mathcal{B}}

\newcommand{\norm}[1]{\lVert #1 \lVert}

\newcommand{\dual}[1]{\big\langle #1 \big\rangle}

\newcommand{\0}{\circ}

\newcommand{\beq}{\begin{equation}}
\newcommand{\eeq}{\end{equation}}

\newcommand{\nn}{\nonumber}

\mathchardef\mhyphen="2D